\documentclass[reqno]{amsart}
\usepackage{graphicx}
\usepackage{amsthm}
\usepackage{url}
\usepackage{multirow}
\usepackage{color}
\usepackage[foot]{amsaddr}

\def\C{{\mathbb C}}
\def\R{{\mathbb R}}
\def\Z{{\mathbb Z}}
\def\N{{\mathbb N}}
\DeclareMathOperator{\real}{{\rm Re}}
\DeclareMathOperator{\imag}{{\rm Im}}

\newtheorem{teo}{Theorem}
\newtheorem{defi}{Definition}
\newtheorem{lema}{Lemma}
\newtheorem{prop}{Proposition}
\newtheorem{remark}{Remark}

\begin{document}

\markboth{A. Bel and W. Reartes and A. Torresi}{Bifurcations in DDEs: an algorithmic approach in frequency domain}

\title[Bifurcations in DDEs: an algorithmic approach in frquency domain]{Bifurcations in Delay Differential Equations: an algorithmic approach in frequency domain}

\author{A. Bel}
\author{W. Reartes}
\author{A. Torresi}
\address{Departamento de Matem\'atica\\
        Universidad Nacional del Sur\\ 
	Av. Alem 1254, 8000 Bah\'{\i}a Blanca, Argentina. }
\email{andrea.bel@uns.edu.ar, walter.reartes@gmail.com}
\maketitle
\begin{abstract}
In this work we study local oscillations in delay differential equations with a frequency domain methodology. The main result is a bifurcation equation from which the existence and expressions of local  periodic solutions can be determined. We present an iterative method to obtain the bifurcation equation up to a fixed arbitrary order. It is shown how this method can be implemented in symbolic math programs.
\end{abstract}



%
\section{Introduction}
A frequency domain approach to study  bifurcations of local periodic solutions of differential equations was initially  presented in \cite{mees79a,mees79b, moiola96}. The notation and main results for this methodology are developed in detail in \cite{mees81}. The frequency domain method, in its different implementations, combines theory of feedback control systems, harmonic balance method and Nyquist stability criterion to find local oscillations in systems of ordinary differential equations. These methods have proven to be useful and interesting for engineers working in control theory.

Some original results in the study of oscillatory solutions for delay differential equations, applying the frequency domain methodology, were presented in the book of Moiola and Chen [1996].
Thereafter, this theory was applied and extended \cite{itovich09} and generalized in \cite{gentile12}, to include more general delayed systems. All these results, however, only consider a harmonic balance approximation up to eight order. Various applications of the frequency domain methodology can be found in \cite{gentile14,yucao07, yu08}.

In the framework of the frequency domain we consider
an harmonic balance of a given order, use a reduction method and obtain an algorithmic process to find coefficients for a bifurcation equation of periodic solutions up to the given order. To analyze this bifurcation equation we apply singularity theory \cite{golubitsky85}. This allows us to classify different bifurcation scenarios of periodic solutions. In particular, we characterize regions of the parameter space in which there are multiplicity of periodic solutions.

In Section \ref{delayed} we reformulate a delayed systems with an input-output representation and we propose a transfer function suitable for applications in the theory of linear feedback systems \cite{mees81}. In Subsection \ref{higher} we prove the main result of the paper and obtain a bifurcation equation for local periodic orbits of general delayed systems. Also, we present an algorithm to find the coefficients of the bifurcation equation and an approximation of periodic solutions up to a fixed order. The algorithm can be implemented in symbolic computation programs such as Mathematica \cite{wolfram} or {\sc Maxima} \cite{maxima}. In Section \ref{analysis}, we show how to analyze the bifurcation equation using tools of singularity theory, and we determine different conditions to find all cycle bifurcations of codimension less than or equal to two. Finally, in Section \ref{examples} we illustrate the proposed method by means of two well-known examples. In our opinion these examples show the potentiality of the proposed methodology.

\section{Delayed systems in the frequency domain}\label{delayed}

We consider an autonomous $n$-dimensional non-linear system of the form
\begin{equation}\label{dfsist0}
x'(t)=f(x(t), x(t-\tau),\mu),
\end{equation}
where $x(t)\in \mathbb{R}^n,$ $x'=dx/dt,$ $\tau>0$ is a time delay, $\mu\in \R$ is the bifurcation parameter and $f:\mathbb{R}^{2n}\times\mathbb{R}\to \mathbb{R}^n$ is a non-linear $C^k$ function with $k>3.$ The system could have more parameters which we call auxiliary.

In an input-output representation we can write the previous system as the following feedback system \cite{mees81}
\begin{equation}\label{dfsist1}
\left\{ \begin{array}{ccl}
x'(t) & = & A_0 \, x(t) + A_1 \, x(t-\tau) + B \, g(y(t),y(t-\tau),\mu),\\
y(t) & = & - C x(t),\\
\end{array} \right.
\end{equation}
with $y(t)\in \R^m,$ $A_{0,1}\in \mathbb{R}^{n\times n},$ $B\in \mathbb{R}^{n\times p},$ $C\in \mathbb{R}^{m\times n},$ and the function $g:\mathbb{R}^{2 m}\times \mathbb{R}\to \mathbb{R}^p$ is defined such that the following holds
\begin{equation}
f(x(t),x(t-\tau),\mu) = A_0 x(t)+ A_1 x(t-\tau) + B\, g(-C x(t), -C x(t-\tau),\mu).
\end{equation}

Here, the value of $g$ represents an input variable which depends on the output variable $y$ and the bifurcation parameter $\mu.$ If $m \le n,$ the dimension of the system is reduced, simplifying its study.

Applying the Laplace transform to system \eqref{dfsist1} and omitting the initial condition effects, we obtain the following expression in the frequency domain
\begin{equation}\label{dfslap}
\mathcal{L}[y(t)](s)= - G(s,\mu,\tau) \mathcal{L}[g(y(t),y(t-\tau),\mu)](s),
\end{equation}
where $G$ is a $m\times p$ matrix defined by
\begin{equation}
G(s,\mu,\tau) = C (sI-A_0 - A_1 e^{-s\tau})^{-1} B.
\end{equation}
$G$ is the transfer function associated to the realization $(A_0,A_1,B,C).$ The realization must be controllable and observable, that is, it must be minimal \cite{morse76,sename01}.

\begin{remark}
The function $g$ contains non-linearities of the system, but it may also contain linear terms related to different realizations.
\end{remark}

With the new representation, the state variable $x$ is deleted and the system is described with input and output variables only. This has advantages when we work with systems of large dimensions but with small number of inputs and outputs. The problem using feedback approach results in a fixed point problem \cite{mees81}.

The equilibria $\hat{y}$ of system \eqref{dfsist1} verify the equation
\begin{equation}
\mathcal{L}[y](s)= - G(0,\mu,\tau) \mathcal{L}[g(y,y,\mu)](s).
\end{equation}
From uniqueness of the Laplace antitransform it results
\begin{equation}
y= - G(0,\mu,\tau) g(y,y,\mu).
\end{equation}
Each equilibrium $\hat{x}$ of system \eqref{dfsist0} corresponds to a solution $\hat{y}$ of the above equation.

Let $z_1$ and $z_2$ be the  first and second variables of the function $g,$ and
\begin{equation}\label{opderiv}
D_1 =\left.\frac{\partial g}{\partial z_1}\right|_{(\hat{y},\hat{y},\mu)},\quad D_2 =\left.\frac{\partial g}{\partial z_2}\right|_{(\hat{y},\hat{y},\mu)},\quad  D_{k_1 k_2}=\left.\frac{\partial^2 g}{\partial z_{k_1} \partial z_{k_2}}\right|_{(\hat{y},\hat{y},\mu)}, \ldots,
\end{equation}
for  $1\le k_1,k_2\le 2,$ the derivatives evaluated at the equilibrium $(\hat{y},\hat{y},\mu)$.
Then, if $g$ is a $C^{k}$ function, we can write its Taylor approximation $\sum_{j=0}^{k} F_j(z)$ as
\begin{equation}\label{taylor}
F_j(z)=\frac{1}{j!}\sum_{k_1,\ldots,k_j=1}^{2} D_{z_{k_1} \ldots z_{k_j}} (z_{k_1}-\hat{z}_{k_1}, \ldots, z_{k_j}-\hat{z}_{k_j}),
\end{equation}
where $z=(z_1,z_2)\in \R^{2n}.$

In particular, linearizing \eqref{dfslap} around the equilibrium $\hat{y}$ we obtain
\begin{equation}
\begin{array}{rcl}
\displaystyle\mathcal{L}[y-\hat{y}](s) &=& -\displaystyle G(s,\mu,\tau) \mathcal{L}\left[D_1 (y-\hat{y})+D_2 (y_{\tau}-\hat{y})\right](s)\vspace{.2cm}\\
&=& -\displaystyle G(s,\mu,\tau) \left( D_1 + D_2 e^{-s\tau}\right) \mathcal{L}[y-\hat{y}](s),
\end{array}
\end{equation}
From this equation we have the following definition.

\begin{defi}\label{defiftl}
The linear transfer function associated to system \eqref{dfsist1}, is the matrix of order $m\times m$ defined by
\begin{equation}\label{gj}
GJ(s,\mu,\tau)=G(s,\mu,\tau)(D_1+D_2 e^{-s\tau}),
\end{equation}
with $D_1$ and $D_2$ defined in \eqref{opderiv}.
\end{defi}

\begin{remark}
In \cite{gentile12} another transfer function is defined, its dimension is bigger than dimension of the matrix defined above.
\end{remark}

Considering this transfer function we can use the frequency domain methodology to find periodic solutions from the original system in the same way as the method used in ordinary differential equations \cite{mees81,moiola96}.

The characteristic functions $\lambda_1(s,\mu,\tau),\ldots,\lambda_r(s,\mu,\tau),$ with $1\leq r\leq m,$ are solutions of
\begin{equation}\label{detcar}
\det(\lambda I-GJ(s,\mu,\tau))=0.
\end{equation}
If a root of the characteristic equation of the non-linear system \eqref{dfsist0} associated to an equilibrium $\hat{x}$ takes the complex value  $i\omega_0$ when $\mu=\mu_0$ and $\tau= \tau_0,$ the corresponding eigenvalue of $GJ(s,\mu,\tau)$ takes the value  $-1$ when $\mu=\mu_0,$ $\tau=\tau_0$ and $s=i\omega_0.$ Thus, we have the following lemma.

\begin{lema}\label{lemma1}
The first necessary condition for the system \eqref{dfsist0} undergoes an Andronov-Hopf bifurcation around $\hat{x}$ at $\mu=\mu_0,$  $\tau=\tau_0$ with critic frequency $\omega=\omega_0\ne 0,$ is that a simple characteristic function $\hat{\lambda}(i\omega,\mu,\tau)$ of $GJ(s,\mu,\tau)$ exists, such that
\begin{equation} \label{hopfc}
\hat{\lambda}(i\omega_0,\mu_0,\tau_0)=-1.
\end{equation}
$\mu=\mu_0,$ $\tau=\tau_0$ and $\omega=\omega_0$ are called critical values.
\end{lema}

In Section \ref{analysis} we will obtain sufficient conditions to ensure the existence of Andronov-Hopf bifurcation as well as other bifurcations of periodic solutions of codimension less than or equal to two. These conditions are obtained by studying the bifurcation equation in the frequency domain. The bifurcation equation will be calculated in the following subsection.

\subsection{High order bifurcation equation}\label{higher}

We will now obtain a theorem for delay differential equations of the form \eqref{dfsist0}. It is an extension of a recent result in the case of ordinary differential equations \cite{torresi12}.

This theorem gives us a bifurcation equation of non-trivial periodic orbits that is very useful to determine the existence and expression of periodic solutions. The study of the bifurcation equation expressed in suitable coordinates and using singularity theory \cite{golubitsky85}, allows us to determine existence and multiplicity of cycles. Furthermore, the theorem gives us analytical expressions for periodic solutions.

Based on the ideas in the proof, we propose an iterative method to make high order calculations for both the bifurcation equation and the approximated periodic solutions. 

\begin{prop}\label{teodfoe}
Consider the system \eqref{dfsist0} with a minimal realization of the form \eqref{dfsist1}, $g \in C^{2q+1}$ for a fixed $q\in \N$, and the transfer function $GJ$ defined in \eqref{gj}. Let $\hat{y}$ be an equilibrium of \eqref{dfsist1} and suppose that critical values exist such that a simple characteristic function $\hat{\lambda}(s,\mu,\tau)$ of $GJ$ verifies \eqref{hopfc}. If $v$ is a eigenvector associated to $\hat{\lambda}$, then the non-trivial solutions $(v,\omega,\mu,\tau)$ near $(0,\omega_0,\mu_0,\tau_0)$ of the equation 
\begin{equation}\label{eqorig}
\phi(v,\omega,\mu,\tau)=0,
\end{equation}
are in one-to-one correspondence with non-trivial periodic solutions of the system  \eqref{dfsist0}, with frequency close to $\omega_0$ and small amplitude.

The function $\phi$ is defined in the proof of the proposition. The equation \eqref{eqorig} is a $q$-th order bifurcation equation of periodic orbits in the frequency domain.
\end{prop}

\begin{proof}
We consider the following approximation of order $2q$ of a periodic solution with frequency $\omega$
\begin{equation}\label{yapp}
\tilde{y}(t)= \hat{y}+\sum_{j=-2q}^{2q} a_j e^{ij\omega t},
\end{equation}
with $a_j\in \mathbb{C}^m,$ $a_{-j}={ \overline{a}_j},$ for $-2q\leq j\leq 2q$. Let $a$ be the vector defined by $a=(a_{-2q}, ..., a_{2q}).$

The function $g$ evaluated at $\tilde{y}(t),$ $\tilde{y}(t-\tau)$, is a periodic function with frequency $\omega$. To find its expansion we consider the Taylor approximation $\sum_{n=0}^{2q+1} F_n$ (see notation in \eqref{taylor}). Then we can write the following approximation of order $2q$ of the function $g$ evaluated in the approximated periodic solution
\begin{equation}
\tilde{g}(\tilde{z},\mu)= g(\hat{z},\mu)+F_1(\tilde{z})+\sum_{j=-2q}^{2q} c_j(a,\mu,\tau) e^{ij\omega t},
\end{equation}
where $\hat{z}=(\hat{y},\hat{y}),$ $\tilde{z}(t)=(\tilde{y}(t),\tilde{y}(t-\tau)),$ and $c_j$ is the $j$-th Fourier coefficient given by
\begin{equation}\label{coefg}
c_j(a,\mu,\tau) = \sum_{n=2}^{2q+1}\frac{1}{n!} \sum_{k_1,\ldots,k_n=1}^2  D_{k_1\ldots k_n} \sum_{\stackrel{\scriptstyle j_1+\cdots+j_n=j}{-2q\le j_1,...,j_n \le 2q}} e^{-i S \omega \tau} (a_{j_1},\ldots,a_{j_n}),
\end{equation}
where $S=\sum_{i=1}^{n}(k_i-1)j_i.$

Introducing $\tilde{y}$ and $\tilde{g}$ in \eqref{dfslap} and using the harmonic balance method \cite{mees79a}, we obtain the equations
\begin{equation}\label{dfajs}
\left(GJ(ij\omega,\mu,\tau)+I \right) a_j + G(ij \omega,\mu,\tau) c_j(a,\mu,\tau)=0,
\end{equation}
for $-2q\leq j \leq 2q.$ We want to solve this system for $a= (a_{-2q}, \ldots,a_{-1},a_0,a_1,\ldots, a_{2q})$ near the origin. For simplicity we define $L_j=GJ(ij\omega,\mu,\tau)+I.$

If $j\neq 1,$ from the hypothesis, the operator $L_j$ is invertible at the critical values $\omega=\omega_0,$ $\mu=\mu_0$ and $\tau=\tau_0.$ Therefore, we can ensure that $a_j$ can be computed for all $j\neq1$ in terms of $a_1$, that is, $a=a(a_1)$ for values of $\omega,$ $\mu$ and $\tau$ near the critical values.

If $j=1,$ we define
\begin{equation}
\Phi(a_1,\omega,\mu,\tau)=L_1 a_1 + G(i\omega, \mu,\tau) c_1(a(a_1),\mu,\tau).
\end{equation}
In this case, the operator $L_1$ is not invertible at the critical values. To solve the equation
\begin{equation}\label{dfeqaf1}
\Phi(a_1,\omega,\mu,\tau)=0,
\end{equation}
we will use projections in appropriate spaces.

We define the operator
\begin{equation}\label{hatl}
\hat{L}=GJ(i\omega,\mu,\tau) -\hat{\lambda}(i\omega,\mu,\tau) I.
\end{equation}
In a neighborhood of the critical values $\dim({\rm ker}(\hat{L}))=1.$ It is possible choose complementary spaces such that $\C^m={\rm ker}(\hat{L})\oplus M$ and $\C^m = N\oplus{\rm range}(\hat{L}).$

Let $Q: \mathbb{C}^m\to N$ be the projection orthogonal to ${\rm range}(\hat{L})$. The complementary projection $I-Q$ has kernel $N$ and acts as the identity on ${\rm range}(\hat{L}).$

Let $v\in {\rm ker}(\hat{L})$ and $a_1\in\mathbb{C}^m,$ then we can write $a_1=v+v^{\bot}$ for some $v^{\bot}\in M$. Thus, $\Phi(a_1,\omega,\mu,\tau)=0$ iff
\begin{eqnarray}
Q\, \Phi(v+v^{\bot},\omega,\mu,\tau) = 0,\label{eqfii} \\
(I-Q)\, \Phi(v+v^{\bot},\omega,\mu,\tau) = 0\label{eqfi}.
\end{eqnarray}

Since $\hat{L}:M\to {\rm range}(\hat L)$ is invertible and from the hypothesis $\hat{L}= GJ(i\omega_0,\mu_0,\tau_0)+I$, we can apply the implicit function theorem to ensure that equation \eqref{eqfi} has an unique solution for $v^{\bot}=v^{\bot}(v,\omega,\mu,\tau)$ near $(0,\omega_0,\mu_0,\tau_0)$.

By replacing $v^{\bot}$  in \eqref{eqfii} we have
$Q\, \Phi(v+v^{\bot}(v,\omega,\mu,\tau),\omega,\mu,\tau)=0.$ Then, the expression of the $q$-th order bifurcation equation is
\begin{equation}\label{bifred}
\phi(v,\omega,\mu,\tau)=Q L_1 (v+v^{\bot})+ Q G(i\omega, \mu,\tau)\, c_1\left(a(v+v^{\bot}),\mu,\tau\right) =0,
\end{equation}
where $v^{\bot}=v^{\bot}(v,\omega,\mu,\tau).$

Thus, in a neighborhood of  $(0,\omega_0,\mu_0,\tau_0),$ each non-trivial solution $(v,\omega,\mu,\tau)$ of the above equation corresponds to a periodic solution of the form \eqref{yapp} with small amplitude and frequency close to $\omega_0$.
\end{proof}

To solve the bifurcation equation \eqref{eqorig} it is customary to choose coordinates on the spaces ${\rm ker}(\hat L)$ and $N.$ The operator $Q$ depends on these choices, so we can obtain different expressions, but they are equivalent from the point of view of the theory of bifurcation.

In the next theorem we choose suitable coordinates on the spaces involved and we obtain from \eqref{eqorig} a reduced bifurcation equation easier to solve. 

\begin{teo}\label{teoppal}
Consider the system \eqref{dfsist0} with a minimal realization of the form \eqref{dfsist1}. Suppose that the conditions in Proposition \ref{teodfoe} are verified, and that $||v||=1.$ Then, there are functions $\xi_k=\xi_k(\omega,\mu,\tau),$ for $1\leq k\leq q$, such that the $q$-th order bifurcation equation of periodic orbits in frequency domain, expressed in coordinates, results
\begin{equation} \label{dfecbif}
\hat{\lambda}(i \omega,\mu,\tau)+1+\sum_{k=1}^{q} \theta^{2k} \xi_k(\omega,\mu,\tau)=0,
\end{equation}
for small $\theta\in\R.$

The non-zero solutions of \eqref{dfecbif} are in one-to-one correspondence with the small amplitude  periodic solutions of the system \eqref{dfsist0}.

Moreover, each solution $(\theta,\omega,\mu,\tau)$ near the critical values of the bifurcation equation above corresponds to a small amplitude periodic solution of the form
\begin{equation}\label{yexp}
y(t)= \hat{y} + \sum_{k=-2q}^{2q} a_k(\theta,\omega,\mu,\tau) e^{i k \omega t}+ \mathcal{O}(\theta^{2q+1}).
\end{equation}

The expressions of $a_k(\theta,\omega,\mu,\tau)$ for $-2q\leq k\leq 2q,$ and $\xi_k(\omega,\mu,\tau)$ for $1\leq k\leq q,$ are obtained in the proof of the theorem.
\end{teo}

\begin{proof}
As in Proposition \ref{teodfoe}, let $v\in {\rm ker}(\hat L)$ be the eigenvector associated to $\hat{\lambda}.$ Considering $||v||=1,$ we can write $a_1=\theta v+v^{\bot},$ with $\theta \in \R$ and $v^{\bot}\in M$. Replacing in equation \eqref{bifred} results
\begin{equation}
Q L_1 (\theta v+v^{\bot}) + QG(i\omega, \mu,\tau)\, c_1\left(a(\theta v+v^{\bot}),\mu,\tau\right)=0.
\end{equation}

For  $w\in {\rm ker} (\hat{L}^{T})={\rm range}(\hat{L})^{\bot}$ (i.e., $w^{T}z=0,$ for $z\in {\rm range}(\hat{L})$) verifying $w^T v=1$, we can write the projection $Q: \mathbb{C}^m\to N$, as $Q(u)=(w^T  u) v,$ for $u \in\mathbb{C}^m$. Then, from the above equation we obtain the bifurcation equation in coordinates
\begin{equation}\label{eqcoor}
(\hat{\lambda}(i\omega,\mu,\tau)+1)(\theta+w^T v^{\bot})  + w^T G(i\omega, \mu,\tau)\,\, c_1\left(a(\theta v+v^{\bot}),\mu,\tau\right) = 0.
\end{equation}

Thus, for $-2q \leq j\leq 2q,$ $j\neq1,$ the equations \eqref{dfajs} and \eqref{eqfi} expressed in coordinates, generate the system
\begin{subequations}
\begin{align}
L_j a_j + G(ij\omega, \mu,\tau) c_j(a(\theta v+v^{\bot}),\mu,\tau) &= 0, \quad \text{if} \quad j\neq 1, \label{dfsaj1}\\
(I-Q)L_1 v^{\bot} + (I-Q) G(i\omega, \mu,\tau) c_1(a(\theta v+v^{\bot}),\mu,\tau) &= 0. \label{dfsaj2}
\end{align}
\end{subequations}

Consider the expansions in $\theta$ of order $2q$ for each coefficient $a_j=\sum_{k=1}^{2q}\theta^k\,a_{j,k}$. If we substitute this expansions in \eqref{dfsaj1} and \eqref{dfsaj2} we obtain expressions of $a_{j,k}$ for each $j$, as coefficients of $\theta^k$. In particular we have,  $a_0 = a_{0,2} \theta^2+a_{0,4}\theta^4+\ldots+a_{0,2q}\theta^{2q},$ $a_1 = v \theta+v^{\bot}= v \theta+ a_{1,3} \theta^3+\ldots+a_{1,2q-1}\theta^{2q-1},$ $a_2 = a_{2,2} \theta^2+a_{2,4}\theta^4+\ldots+a_{2,2q}\theta^{2q},$ $a_3 = a_{3,3} \theta^3+a_{3,5}\theta^5+\ldots+a_{3,2q-1}\theta^{2q-1},
\cdots, $ $a_{2q} = a_{2q,2q} \theta^{2q}.$

Furthermore, introducing the above formulas in \eqref{eqcoor}
and considering the coefficients of $\theta^{2k+1}$ in this expression, we find the functions $\xi_k(\omega,\mu,\tau)$ for $1\leq k \leq q.$ Thus, from \eqref{eqcoor} we obtain the bifurcation equation in coordinates \eqref{dfecbif}.

For each $\mu,$ $\tau$ and $\omega$ near critical values and $\theta$ that solves equation \eqref{dfecbif}, we find a solution of the form \eqref{yexp}, that corresponds to a periodic solution of the original system \eqref{dfsist0}.
\end{proof}

Unlike the graphic Hopf bifurcation theorem, the above theorem allows us to analyze the existence of multiple cycles. In the following section we will show  that studying equation \eqref{dfecbif}, we can determine regions in the parameter space with more than one cycle and manifolds in that space in which the number of cycles changes (at least locally).

For a fixed order $q,$ we can solve equations \eqref{dfsaj1} and  \eqref{dfsaj2} iteratively by increasing the powers of $\theta$ up to $\theta^{2q}$. In the following remark we show the steps to solve the equations considering $ q =1 $. Then we present an algorithm applicable for any order $q$.

\begin{remark}
For $q=1$, the bifurcation equation results
\begin{equation}
\hat{\lambda}(i \omega,\mu,\tau)+1+ \theta^{2} \xi_1(\omega,\mu,\tau)=0.
\end{equation}
and the expression for the solution of order $2$ is
\begin{equation}
y(t)= \hat{y} + \sum_{j=-2}^{2} a_j(\theta,\omega,\mu,\tau) e^{i j \omega t}.
\end{equation}

We need to calculate the expression for $a=(\bar{a}_2, \bar{a}_1,a_0,a_1,a_2)$, this is, the coefficients $a_0 = a_{0,2}\theta^2$, $a_{1}=\theta v$ and $a_2=a_{2,2}\theta^2$.

Step 1. Let $a_1=\theta v$ in $a$, and set the rest of the coefficients to 0. We substitute in  \eqref{dfsaj1}, for $j=0,2,$ and we obtain that the coefficients of $\theta^2$ are $a_{0,2}$ and $a_{2,2}$, respectively.

Besides, we have to calculate $\xi_1$ to determine the bifurcation equation.

Step 2. Let $a_0=a_{0,2} \theta^2 $, $a_1=\theta v$ and  $a_2=a_{2,2} \theta^2$  in $a.$ We replace in \eqref{eqcoor} and obtain  $\xi_1(\omega,\mu,\tau)$ as the coefficient of $\theta^3$.
\end{remark}

For a fixed order $q,$ we summarize the steps to apply the algorithmic approach of the frequency method in Table \ref{mfiter}. In particular, we show how to obtain expressions for the coefficients $\xi_i$ in the bifurcation equation \eqref{dfecbif}. We note that, if the dimension of the output variable $y$ is $1,$ the vectors $v$ and $w$ associated to $\hat{\lambda}$ are equal to $1.$

\begin{table}[t] \centering%
{\footnotesize\caption{Algorithm to apply frequency method for order $q.$}\label{mfiter} \vspace{0.3cm}
\begin{tabular}{l}
\hline\hline\\

{\bf General definitions} \vspace{0.1cm}\\
\begin{tabular}{ll}
 & Let $(A_0,A_1,B,C)$ be a minimal realization of the system \eqref{dfsist0}. \\
 & Calculate $GJ$ in \eqref{gj} and the eigenvalue $\hat{\lambda}$ such that:\\
 & at critical values $\omega_0,$ $\mu_0$ and $\tau_0,$ it results $\hat{\lambda}(i\omega_0,\mu_0,\tau_0)=-1.$\\
 & Calculate the eigenvectors associated to $\hat{\lambda}$: $v$ for $GJ$ and $w$ for $GJ^{\, T}.$ \\
 & Normalize $v$ and $w,$ such that $||v||=1$ and $w^T v=1.$\\
 & Consider $a=(\bar{a}_{2q},\ldots,\bar{a}_1,a_0,a_1,
\ldots,a_{2q}),$ being \\
 & $a_0 = a_{0,2} \theta^2+a_{0,4}\theta^4+ \ldots+a_{0,2q}\theta^{2q},$\\
 & $a_1 = v \theta+v^{\bot}= v \theta+ a_{1,3}\theta^3+\ldots+a_{1,2q-1}\theta^{2q-1},$\\
 & $a_2 = a_{2,2}\theta^2+a_{2,4}\theta^4+ \ldots+a_{2,2q}\theta^{2q},$\\
 & $a_3 = a_{3,3}\theta^3+a_{3,5}\theta^5+ \ldots+a_{3,2q-1}\theta^{2q-1},$ $\ldots $, $a_{2q} = a_{2q,2q} \theta^{2q}.$ \vspace{.3cm}\\
\end{tabular}\\
\hline\\
{\bf Iterative process}\vspace{0.1cm}\\

\begin{tabular}{lll}
&{\small Step 1} &Let $a_1=v\theta$ in $a$ and the rest of coefficients 0. \\
& &  For $k$ from $1$ to $q:$\vspace{0.1cm} \\
& & \begin{tabular}{ll}
       {\small 1.1} & Replace $a$ in \eqref{dfsaj1} for $j=0,2,\ldots, 2k$ and calculate $a_{j,2k}$ as the \\
       & coefficients of $\theta^{2k}.$ Then update vector $a$. If $k=q$ go to Step 2.\\
       {\small 1.2} & Replace $a$ in \eqref{dfsaj2} and \eqref{dfsaj1} to calculate $a_{1,2k+1}$ and\\
       &  $a_{j,2k+1},$ as the coefficient of $\theta^{2k+1}$ in each equation, respectively.\\
       &Then update vector $a$.\\
    \end{tabular}\\

&{\small Step 2}  & Replace $a$ in \eqref{eqcoor} and calculate $\xi_k$ as the coefficient of $\theta^{2k+1},$\\
&&  for $k=1, \ldots,q.$ \vspace{0.1cm}\\

\hline\hline
\end{tabular}
\end{tabular}}
\end{table}

\section{Analysis of the bifurcation equation using singularity theory}\label{analysis}

In Theorem \ref{teoppal} we proved that every solution of the bifurcation equation in the frequency domain \eqref{dfecbif} is associated with a periodic solution of the delayed system \eqref{dfsist0}. In this section we propose two approaches to study the bifurcation equation with singularity theory. In particular, we can analyze all bifurcation of periodic solutions with codimension less than or equal to two. The terminology of the singularity theory used in this section is detailed in \cite{golubitsky85}.

The first approach allows us to determine generalized Andronov-Hopf bifurcations and periodic solutions with small amplitude. The second approach is more suitable to analyze periodic solutions connecting two different Andronov-Hopf bifurcation points.

Hereafter we suppose that system \eqref{dfsist0} has bifurcation parameter $\mu$ and vector of auxiliary parameters $\rho \in \R^r.$ To simplify the notation we consider the delay $\tau$ as an auxiliary parameter.

The stability of periodic solutions is determined using information of the equilibrium stability as well as the direction in which the periodic solutions emerge when the bifurcation parameter varies.

\subsection{Bifurcations of periodic orbits with small amplitude}
We study the bifurcation equation to determine the existence and multiplicity of periodic solutions with small amplitude. Using singularity theory we analyze Andronov-Hopf and generalized Andronov-Hopf bifurcations, and we describe regions of the parameter space in which the system presents up to three periodic orbits.

Let $(\omega_0,\mu_0,\rho)$ be a point in the parameter space such that the hypothesis of Theorem \ref{teoppal} are verified and
\begin{equation}\label{condinodeg}
\left.\left(\real{\hat{\lambda}_{\mu}} \imag{\hat{\lambda}_{\omega}}-\imag{\hat{\lambda}_{\mu}} \real{\hat{\lambda}_{\omega}}\right)\right|_{(\omega_0,\mu_0,\rho)}\neq0.
\end{equation}
Then, the real and the imaginary parts of the bifurcation equation \eqref{dfecbif} form a system that verifies the hypothesis of the implicit function theorem at $(\omega_0,\mu_0,\rho).$ Taking implicit derivatives, for   $\theta\neq 0$ small enough, we have
\begin{eqnarray}
\mu & = & \mu_0 + \mu_1 \theta^2 +  \cdots + \mu_{q} \theta^{2q},\label{ecmu}\\
\omega & = & \omega_0 + \omega_1 \theta^2 + \cdots  + \omega_{q} \theta^{2q},
\end{eqnarray}
where $\mu_{k}=\mu_{k}(\rho)$ and $\omega_{k}=\omega_{k}(\rho),$ for $0 < k\le q.$

Studying equation \eqref{ecmu} we can describe the dynamics associated with periodic solutions of the system \eqref{dfsist0}, in large regions of the auxiliary parameter space. In particular, depending on the order $q,$ it is possible to determine the existence of different bifurcations, for example: Andronov-Hopf bifurcation, generalized Andronov-Hopf bifurcation and fold bifurcation for limit cycles.

We suppose that equation \eqref{ecmu} with $\rho=\rho_0$ is $\Z_2$-equivalent (see \cite{golubitsky85}) to the normal form
\begin{equation} \label{fnhopfgen}
\mu_q(\rho_0)\theta^{2q} - (\mu-\mu_0)=0.
\end{equation}
Then, the coefficients $\mu_k(\rho),$ $1\le k\le q-1,$ verify $\mu_k(\rho_0)=0,$ and in a neighborhood of $\mu=\mu_0$ and $\rho=\rho_0,$ in which $\mu_q(\rho)\neq 0,$ we have the universal unfolding (see \cite{golubitsky85})
\begin{equation}
\mu_q(\rho) \theta^{2q} - (\mu-\mu_0) + \mu_1(\rho) \theta^2 + \mu_2(\rho) \theta^4+ \cdots + \mu_{q-1}(\rho)\theta^{2(q-1)}=0.
\end{equation}

In the next, we detail the unfoldings for $q=1,2,3.$ In each case, we indicate the transition varieties (see \cite{golubitsky85}) defined in parameter space $(\mu_1,\ldots,\mu_{q-1})$. This varieties determine regions in the parameter space in which the system has persistent bifurcation diagrams (see \cite{golubitsky85}). The transition varieties are associated with changes in the number of periodic solutions with small amplitude.

To determine the stability of the periodic solutions, we suppose the equilibrium is asymptotically stable for values of the bifurcation parameter less than the critic value $\mu=\mu_0$. In the figures we indicate stable (unstable) equilibria and cycles with continuous (dashed) line.

\begin{itemize}
\item Case $q=1.$ Andronov-Hopf bifurcation.

The normal form is
\begin{equation}
\mu_1(\rho) \theta^{2} - (\mu-\mu_0) =0.
\end{equation}
If $\mu_1(\rho)>0$ ($\mu_1(\rho)<0$) the bifurcation is supercritical (subcritical).

\item Case $q=2.$ Generalized Andronov-Hopf bifurcation (Bautin bifurcation).

The universal unfolding of \eqref{fnhopfgen} results
\begin{equation} \label{bautinuu}
\mu_2(\rho) \theta^{4} - (\mu-\mu_0) + \mu_1(\rho) \theta^2 =0.
\end{equation}
We suppose $\mu_2(\rho)>0$ in a neighborhood of $\rho_0$ (the case $\mu_2(\rho)<0$ is similar). The transition variety is $H_0:\mu_1(\rho)=0.$ In Figure \ref{casoq2} we show the two different persistent bifurcation diagrams. If $\mu_1(\rho)>0$ a stable periodic solution exists for $\mu>\mu_0,$ whereas that if $\mu_1(\rho)<0$ a fold bifurcation of periodic solutions is observed, the cycle with smaller amplitude is unstable and the other cycle results stable.

\begin{figure}
\begin{center}
\includegraphics[width=0.4\textwidth]{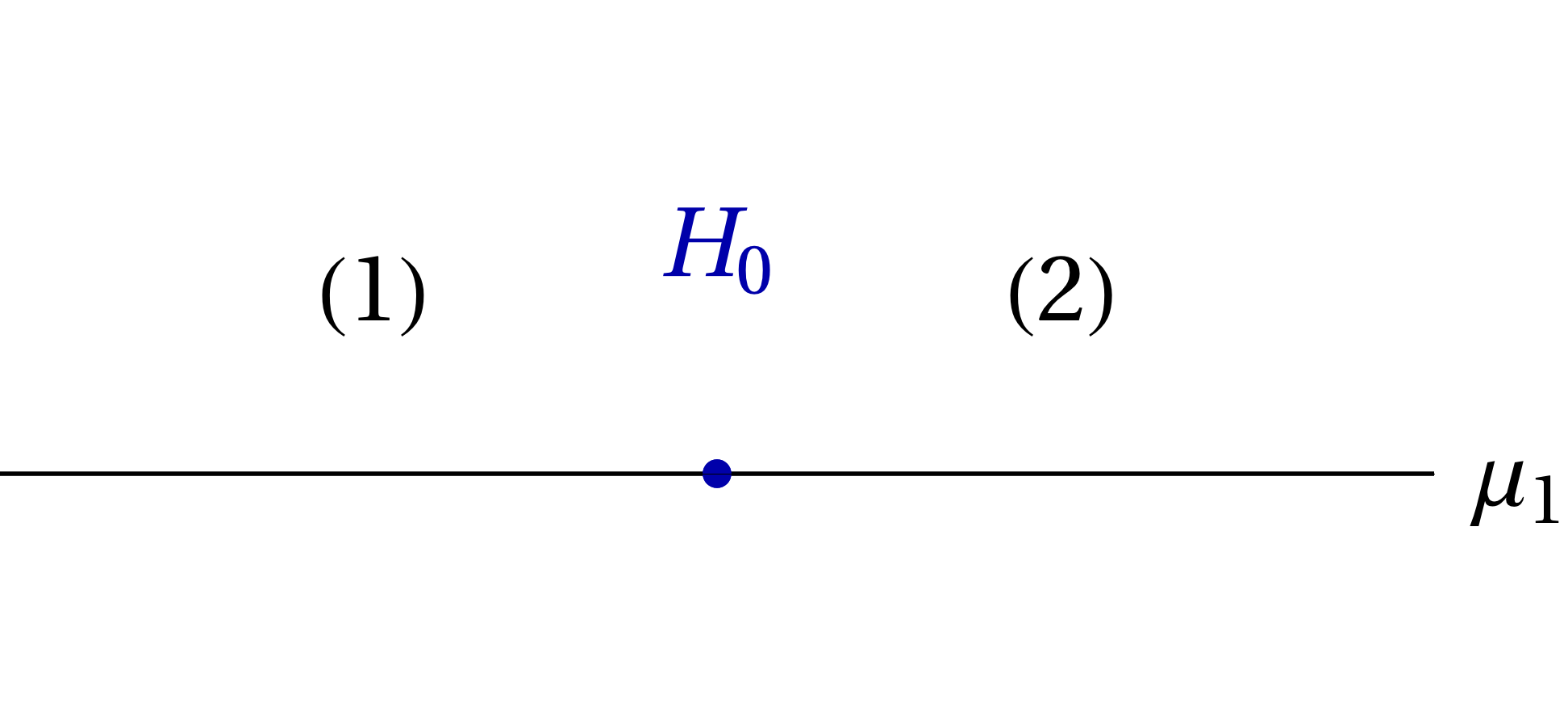} \hspace{.5cm}
\includegraphics[width=0.4\textwidth]{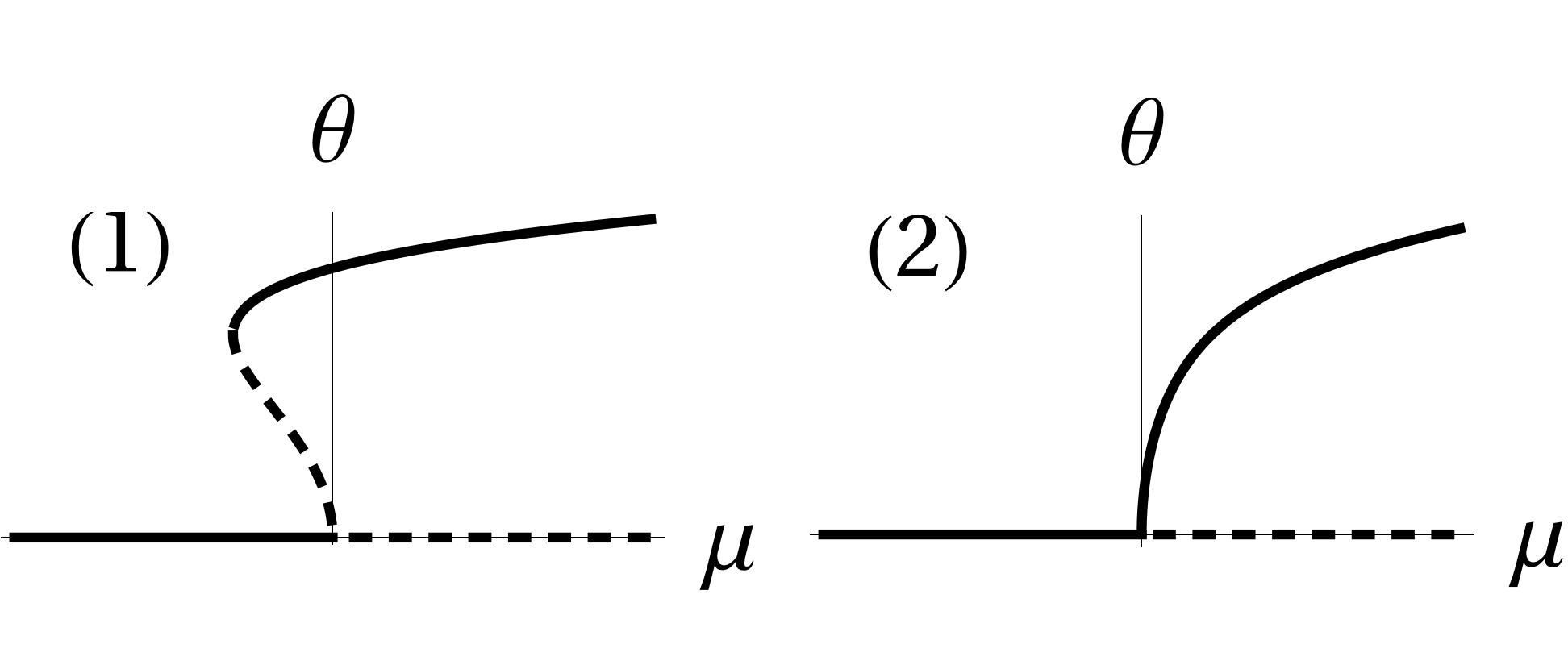}
\caption{Case $q=2.$ Left: Transition variety $H_0$ in the parameter space $\mu_1$. Right: Persistent bifurcation diagrams for the values indicated in the figure on the left. Stable (unstable) equilibria and cycles in continuous (dashed) line. } \label{casoq2}
\end{center}
\end{figure}

\item Case $q=3.$ Generalized Andronov-Hopf bifurcation.

The universal unfolding is
\begin{equation}\label{uuq3}
\mu_3(\rho) \theta^{6} - (\mu-\mu_0) + \mu_1(\rho) \theta^2 + \mu_2(\rho) \theta^4 =0.
\end{equation}
We suppose $\mu_3(\rho)>0$ in a neighborhood of $\rho_0$ (the case $\mu_3(\rho)<0$ is similar). The transition varieties, in the space $(\mu_1,\mu_2),$ are
\begin{equation}\label{vtq3}
\begin{array}{l}
H_0:\mu_1=0,\\
H_1:3 \mu_1 \mu_3 = \mu_2^2, \quad \mu_2 \mu_3\le 0,\\
D:4 \mu_1 \mu_3 = \mu_2^2, \quad \mu_2 \mu_3 \le0.
\end{array}
\end{equation}
In Figure \ref{casoq3} we plot $H_0,$ $H_1$ and $D,$ and the different persistent bifurcation diagrams obtained in regions of the parameter space limited for this curves. For suitable parameter values we can observe up to three coexistent periodic orbits of small amplitude.

\end{itemize}

\begin{figure}
\begin{center}
\includegraphics[width=0.4\textwidth]{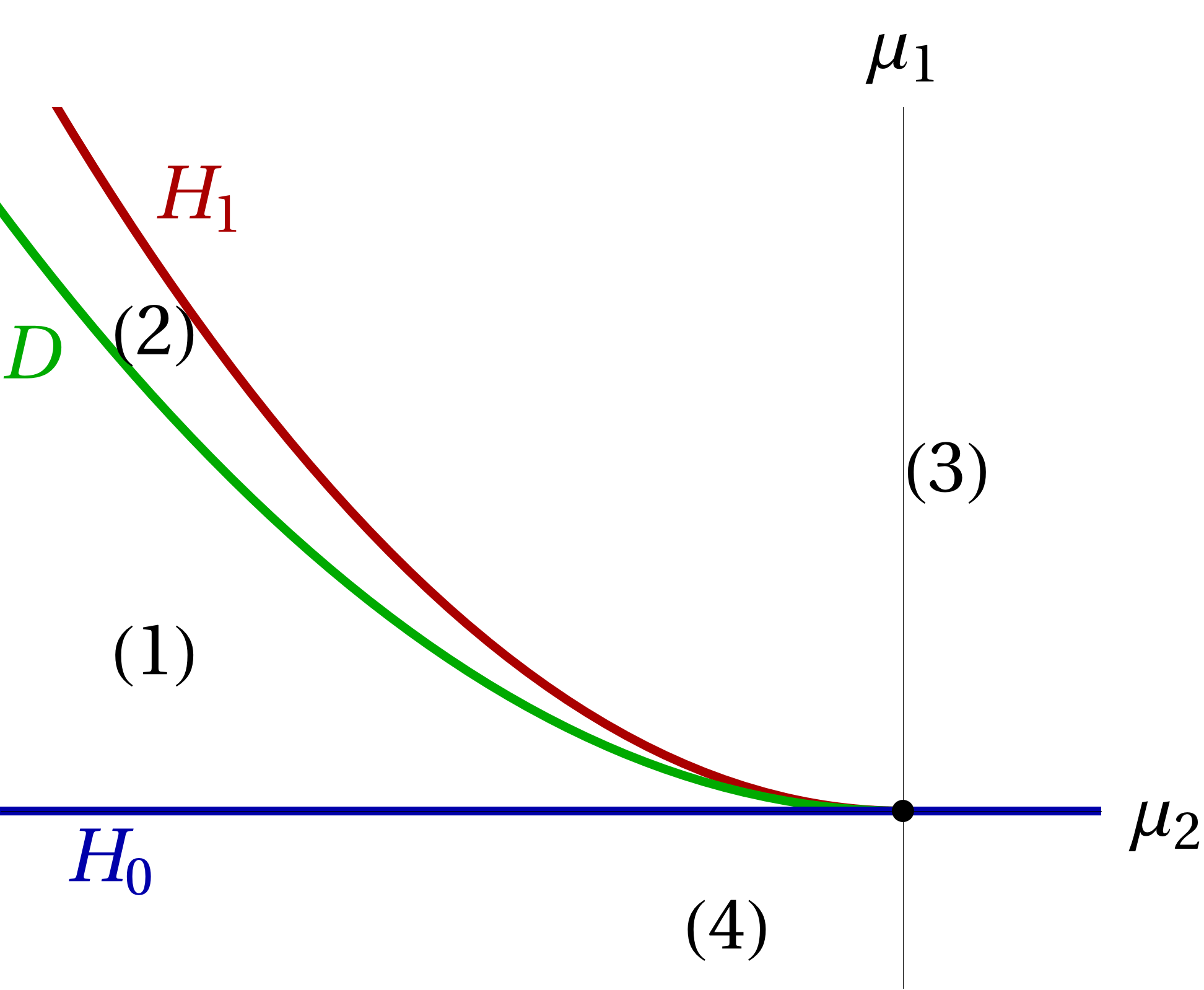} \hspace{.5cm}
\includegraphics[width=0.4\textwidth]{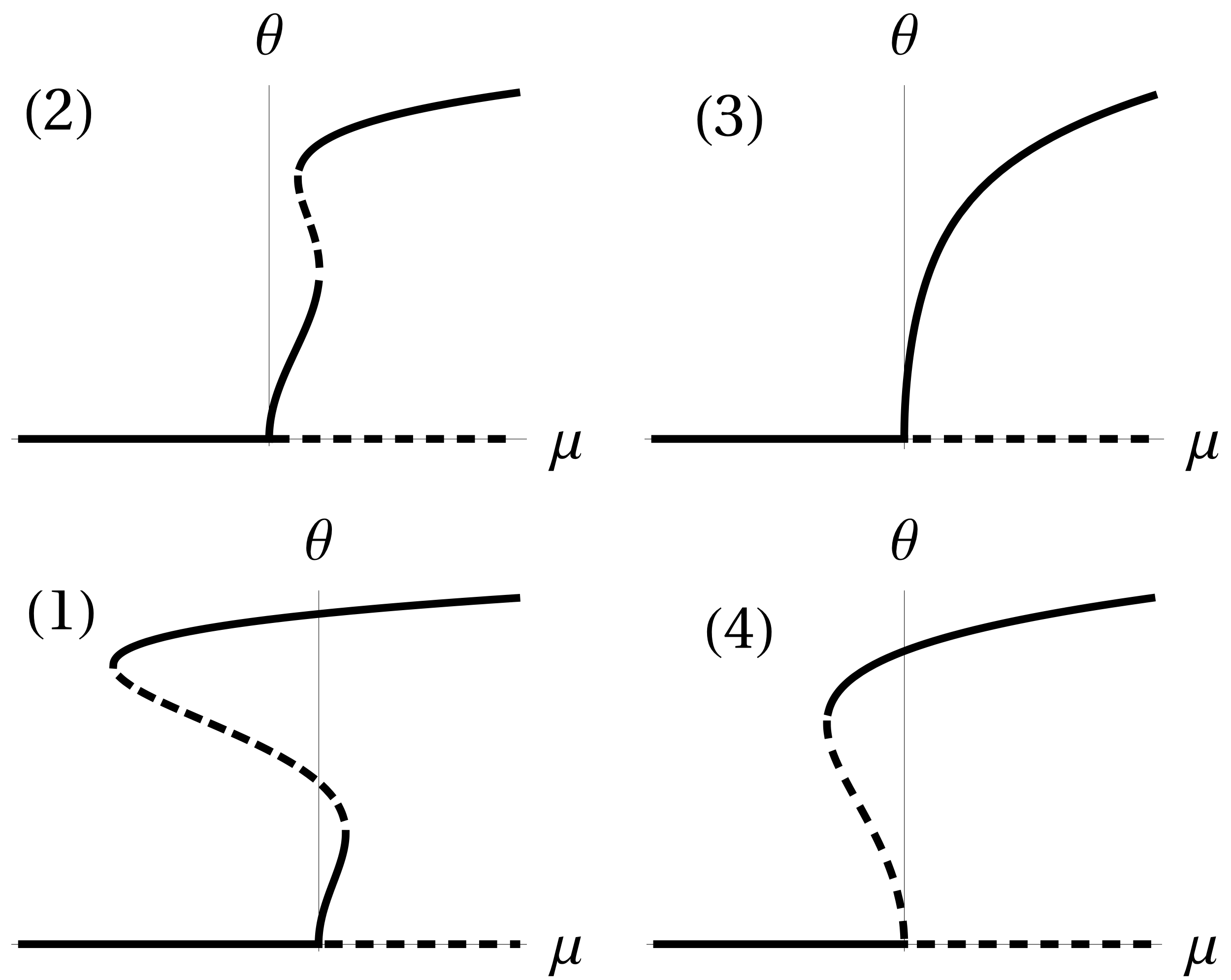}
\caption{Case $q=3.$ Left: Transition varieties $H_0,$ $H_1$ and $D$ defined in \eqref{vtq3}. Right: Persistent bifurcation diagrams for values indicated in the figure on the left. Stable (unstable) equilibria and cycles in continuous (dashed) line.} \label{casoq3}
\end{center}
\end{figure}

\subsection{More bifurcations of periodic orbits}

In this section we present an alternative approach to study periodic solutions with the bifurcation equation \eqref{dfecbif}. Here, we consider the frequency $\omega$ as variable (rather than the amplitude). Unlike the analysis in the previous section, this consideration allows us to study all bifurcations of periodic orbits with normal form of codimension less than or equal to two \cite{golubitsky85}.  
This alternative approach is particularly useful to analyze periodic solutions that connect two Andronov-Hopf bifurcation points.

Considering $z=\theta^2,$ the bifurcation equation results
\begin{equation}
\hat{\lambda}(i \omega,\mu)+1+\sum_{k=1}^{q} z^{k} \xi_k(\omega,\mu)=0.
\end{equation}
By the implicit function theorem, if
\begin{equation}
\left.\left(\real{\hat{\lambda}_{\mu}} \imag{\xi_1}-\imag{\hat{\lambda}_{\mu}} \real{\xi_1}\right)\right|_{(\omega_0,\mu_0,\rho)}\neq0,
\end{equation}
is verified, we can consider $\mu$ and $z$ as functions of $\omega,$ in a neighborhood of $(\omega_0,\mu_0,\rho)$ and for small values of $z.$  To simplify the notation the parameter $\rho$ is considered as a variable of the functions $\mu$ and $z$.

With this notation, the bifurcations described in the above subsection result as follow. Suppose that $\frac{d^k \mu}{d\omega^k}(\omega_0,\rho_0)= 0,$  for $1\leq k\leq q-1,$ and $\frac{d^q \mu}{d\omega^q}(\omega_0,\rho)\neq 0,$ in a neighborhood of $\rho=\rho_0.$ Also, suppose that $\frac{d z}{d\omega}(\omega_0,\rho)\neq 0$ in a neighborhood of $\rho=\rho_0.$
Then, we obtain an equation that relates $\mu$ and $z,$ which is $\Z_2$-equivalent to the normal form
\begin{equation}
\epsilon_q(\rho_0) z^q-(\mu-\mu_0)=0.
\end{equation}
Considering $\epsilon_q(\rho)>0$ for $\rho$ near $\rho_0$ and since $z=\theta^2,$ the universal unfolding for each $q=1,2,3,$ is similar to the one studied in the above subsection.

Now we detail the two remaining normal form with codimension less than or equal to two. Suppose that $\frac{d z^k}{d\omega^k}(\omega_0,\rho_0) = 0,$ for $1\leq k \leq p-1,$ and $\frac{d z^p}{d\omega^p}(\omega_0,\rho)\neq 0$ and $\frac{d \mu}{d\omega}(\omega_0,\rho)\neq 0,$ are verified in a neighborhood of $\mu=\mu_0$ and $\rho=\rho_0.$ Then,  we obtain an equation $\Z_2$-equivalent to the normal form
\begin{equation}\label{normalp}
\epsilon(\rho_0) z+(\mu-\mu_0)^p=0.
\end{equation}
We describe next the universal unfoldings for $p=2,3.$

\begin{itemize}
\item Case $p=2$. \\
  The universal unfolding results
\begin{equation}
  \epsilon(\rho) z + (\mu-\mu_0)^2 + \epsilon_0(\rho) =0.
\end{equation}
In Figure \ref{casoqp1} we plot the persistent bifurcation diagrams for $\epsilon(\rho)>0.$ The transition variety is $B_0:\epsilon_0(\rho)=0.$  If $\epsilon_0(\rho)<0,$ a branch of stable periodic solutions exists that connect two Andronov-Hopf bifurcation points. If $\epsilon_0(\rho)>0$ there is not periodic solutions and the equilibrium is stable.

\begin{figure}
\begin{center}
\includegraphics[width=0.4\textwidth]{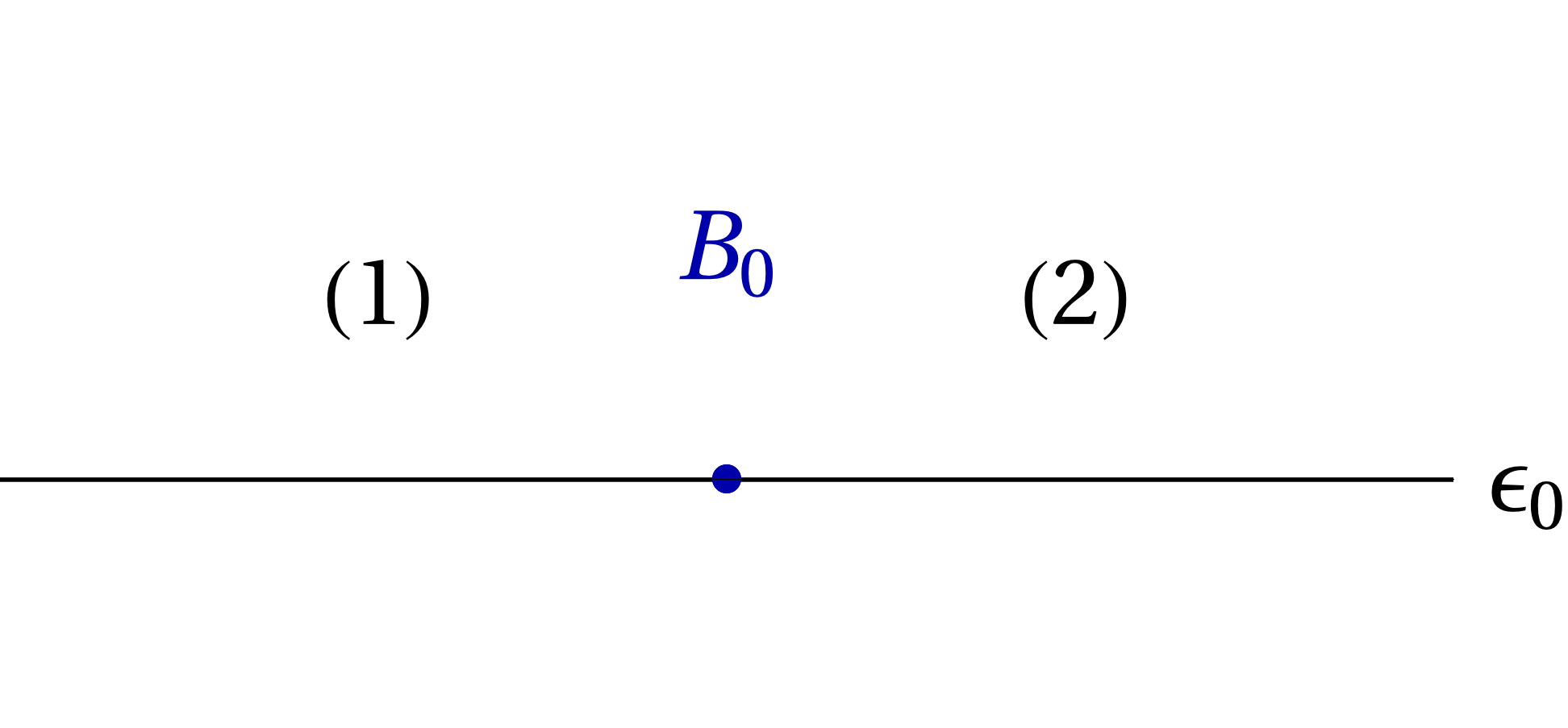} \hspace{.5cm}
\includegraphics[width=0.4\textwidth]{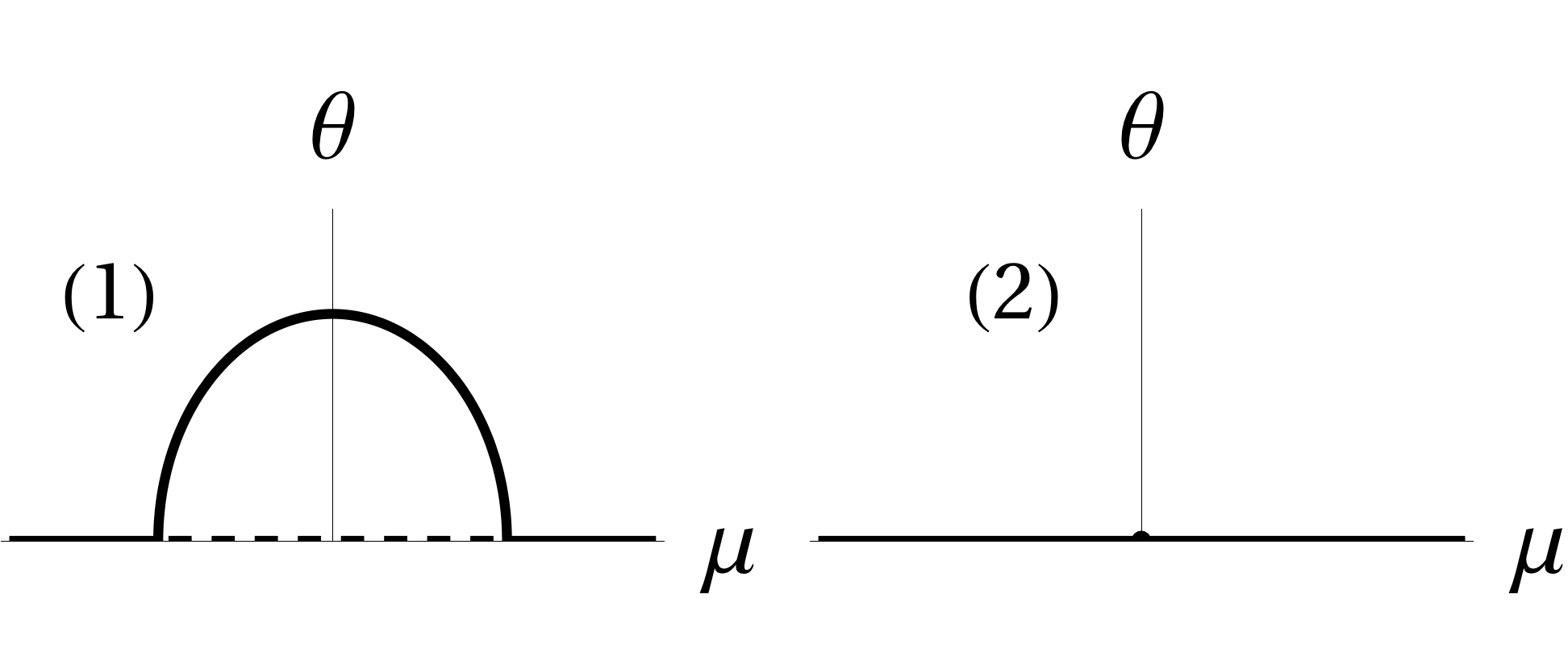}
\caption{Case $p=2.$ Left: Transition variety $B_0.$ Right: Persistent bifurcation diagrams indicated in the figure on the left. Stable (unstable) equilibria and cycles in continuous (dashed) line.} \label{casoqp1}
\end{center}
\end{figure}

\item Case $p=3$. \\
The universal unfolding of \eqref{normalp} is
$$ \epsilon(\rho) z + (\mu-\mu_0)^3 +\epsilon_1(\rho) (\mu-\mu_0) + \epsilon_0(\rho) =0.$$
We suppose that $\epsilon(\rho)>0.$ The transition variety is $B:(\epsilon_1/3)^3=(\epsilon_0/2)^2.$ In Figure \ref{casoqp2} we show the two different persistent bifurcation diagrams. If $(\epsilon_1/3)^3<(\epsilon_0/2)^2,$ the diagram of periodic solutions is topologically equivalent to that obtained in a supercritical Andronov-Hopf bifurcation. If $(\epsilon_1/3)^3>(\epsilon_0/2)^2,$  there is a supercritical bifurcation, but also for lower values of the parameter $\mu,$ there is a branch of stable periodic solutions connecting two bifurcation points.

\begin{figure}
\begin{center}
\includegraphics[width=0.4\textwidth]{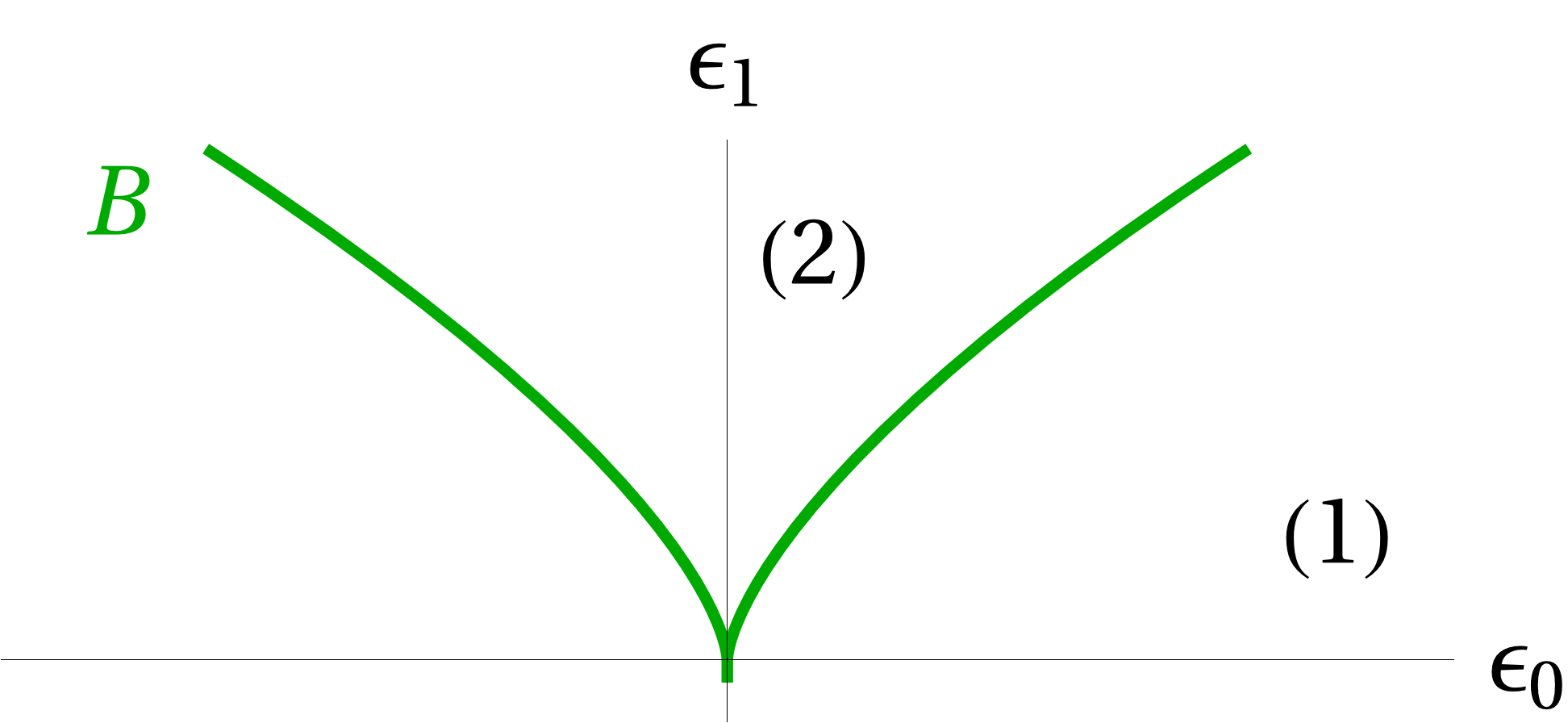} \hspace{.5cm}
\includegraphics[width=0.4\textwidth]{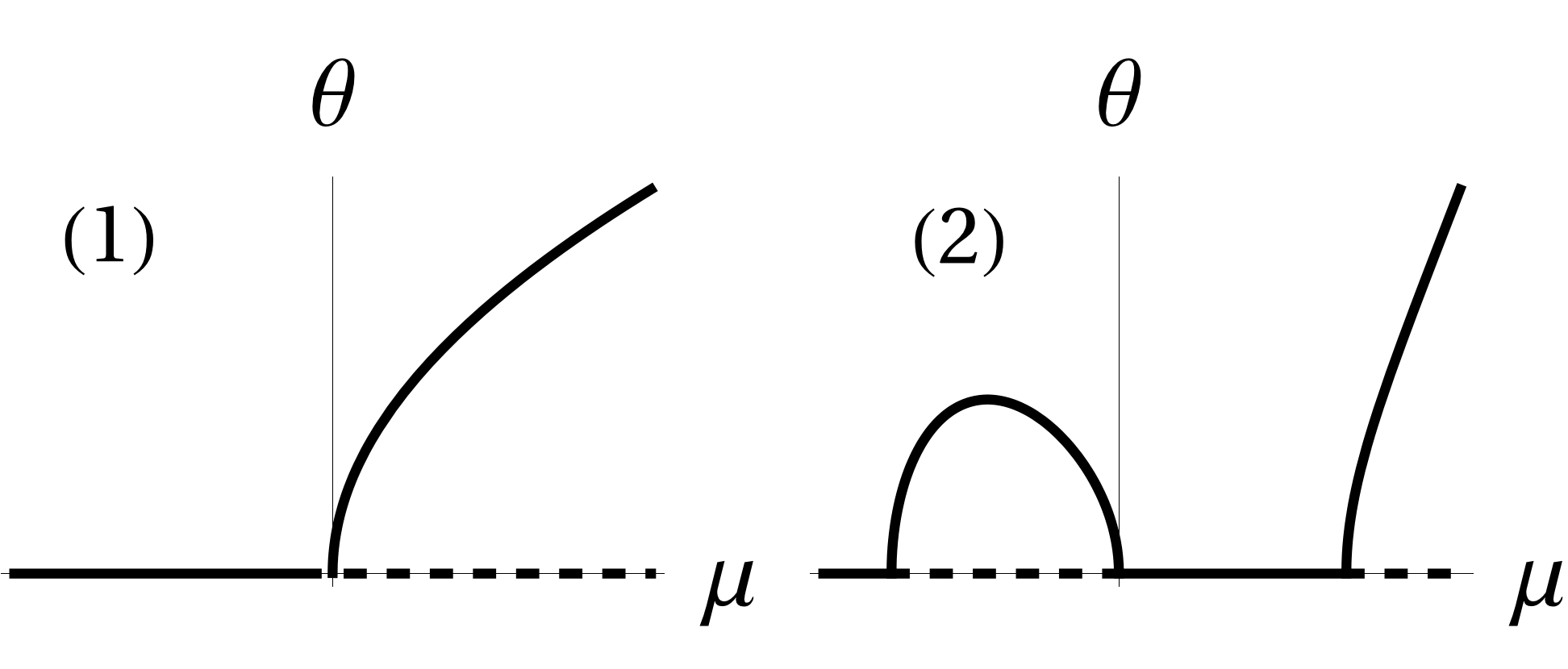}
\caption{Case $p=3.$ Left: Transition variety $B.$ Right: Persistent bifurcation diagrams indicated in the figure on the left. Stable (unstable) equilibria and cycles in continuous (dashed) line. } \label{casoqp2}
\end{center}
\end{figure}

\end{itemize}

The results obtained with the frequency method are local. Thus, the dynamics described in this section correspond with dynamics of the original system in small neighborhoods of $\mu=\mu_0$ and $\rho=\rho_0$ in the parameter space, and for periodic solutions of small amplitude.

\section{Examples}\label{examples}

In this section we study two systems using frequency domain methods in combination with singularity theory. We show in both examples the potentiality of the results obtained in the above sections.
In particular, in the first example we detail how to apply the algorithmic method using Table \ref{mfiter} and we analyze, with the two approaches proposed in Section \ref{analysis}, different bifurcations that this system presents. In the second example, we  determine a Bautin bifurcation in a first order delay differential equation.

\subsection{Bifurcations in a differential equation with delayed feedback}

We consider the system
\begin{multline}\label{just} x'(t)= \Big( \begin{array}{cc} \mu & -1 \\ 1 & \mu \end{array} \Big) x(t) + \left(x_1(t)^2+x_2(t)^2\right) \Big( \begin{array}{cc} 1 & -\gamma \\ \gamma & 1 \end{array} \Big) x (t) - \\
\kappa \Big( \begin{array}{cc} \cos \beta & -\sin \beta \\ \sin \beta & \cos \beta \end{array} \Big)(x(t)- x(t-\tau)) ,
\end{multline}
where $x(t)=(x_1(t),x_2(t))\in \R^2,$ the bifurcation parameter is $\mu\in \R$, the auxiliary parameters are $\gamma,$ $\kappa$ and $\beta,$ being $\gamma>0,$ $\kappa\neq 0$ and $\beta \in \R,$ and $\tau>0$ is the constant delay.

The above system represents a basic form of a subcritical Andronov-Hopf bifurcation with a Pyragas delayed feedback. Similar systems were studied in \cite{fiedler07, gentile12,just07} aiming to make stable the periodic cycle. Presence of feedback induces changes of stability for certain values of parameters, but also generates a variety of bifurcation scenarios in which, among others, we can observe multiplicity of periodic orbits.

We rewrite the equation \eqref{just} using the input-output representation
\begin{equation}\label{dfjust}
\left\{ \begin{array}{ccl}
x'(t) & = &(\mu -\kappa \cos \beta) x(t) + \kappa \cos \beta\,  x(t-\tau) + g(y(t),y(t-\tau)),\\
y & = & - x,
\end{array} \right.
\end{equation}
where $g:\mathbb{R}^4\to \mathbb{R}^2$ is the function defined by
\begin{equation}
g(y(t),y(t- \tau))= \Big( \begin{array}{c} y_2 -\kappa \sin \beta (y_2 - y_{2\tau}) - (y_1^2+y_2^2)(y_1 - \gamma y_2)\\ -y_1 +\kappa \sin \beta (y_1 - y_{1\tau}) - (y_1^2+y_2^2)(\gamma y_1 + y_2) \end{array} \Big),
\end{equation}
being $y_{i\tau}=y_i(t-\tau),$ for $i=1,2$. The realization $((\mu-\kappa \cos\beta),\kappa \cos\beta ,1,1)I_2,$ with $I_2$ the identity matrix of order $2,$ is minimal and its associated transfer function results
\begin{equation}
\displaystyle G(s,\mu,\tau) = \frac{1}{s-\mu+\kappa \cos \beta (1-e^{-s\tau})}\, I_2.
\end{equation}

If $\mu\gamma\neq 1,$ the unique equilibrium of the system is $\hat{y}=(0\,\, 0)^T.$ The transfer matrix for this equilibrium is
\begin{equation}
GJ(s,\mu,\tau)=\frac{1 -\kappa \sin \beta (1 - e^{-s\tau})}{s-\mu+\kappa \cos \beta (1-e^{-s\tau})}\Big( \begin{array}{c c} 0 & 1 \\ -1 & 0\end{array} \Big).
\end{equation}
The characteristic function
\begin{equation}\label{lambdas}
\hat{\lambda}(s,\mu,\tau) = -i \frac{1 -\kappa \sin \beta (1 - e^{-s\tau})}{s-\mu+\kappa \cos \beta (1-e^{-s\tau})},
\end{equation}
verifies Lemma \ref{lemma1} (i.e., take value -1), at critical values
defined by
%
\begin{equation}
\tau =\frac{\pm \arccos(\cos\beta -\frac{\mu}{\kappa})+\beta + 2\pi n}{\omega}, \quad n\in\N,
\end{equation}
being $\omega=1-\kappa \sin \beta \pm \sqrt{\kappa^2-(-\mu +\kappa \cos \beta)^2}.$ The above equation define curves in the $\mu$-$\tau$ space. In Figure \ref{hopfjust} we show some of this curves for fixed $\beta=\pi/4$ and values of $\kappa$ indicated.

For $\kappa$ small enough (for example, $\kappa \in(-0.1,0.1)$), $\hat{\lambda}$ verifies
$\Big.\Big(\frac{\partial {\rm Re}\, \hat{\lambda}}{\partial \mu}\frac{\partial {\rm Im}\, \hat{\lambda}}{\partial \omega} -\frac{\partial {\rm Re}\, \hat{\lambda}}{\partial \omega} \frac{\partial {\rm Im}\, \hat{\lambda}}{\partial \mu}\Big)\Big|_{(\omega_0,\mu_0,\tau_0)}\neq0,$
at critical values $(\omega_0,\mu_0,\tau_0)$. Then the system \eqref{dfjust} has a local bifurcation of periodic solutions. 
Increasing the value of $\kappa$ we can observe that values of $\tau$ exist in which there are more than one Andronov-Hopf bifurcation point (see Fig. \ref{hopfjust} right).

\begin{figure}
\begin{center}
\includegraphics[width=0.3\textwidth]{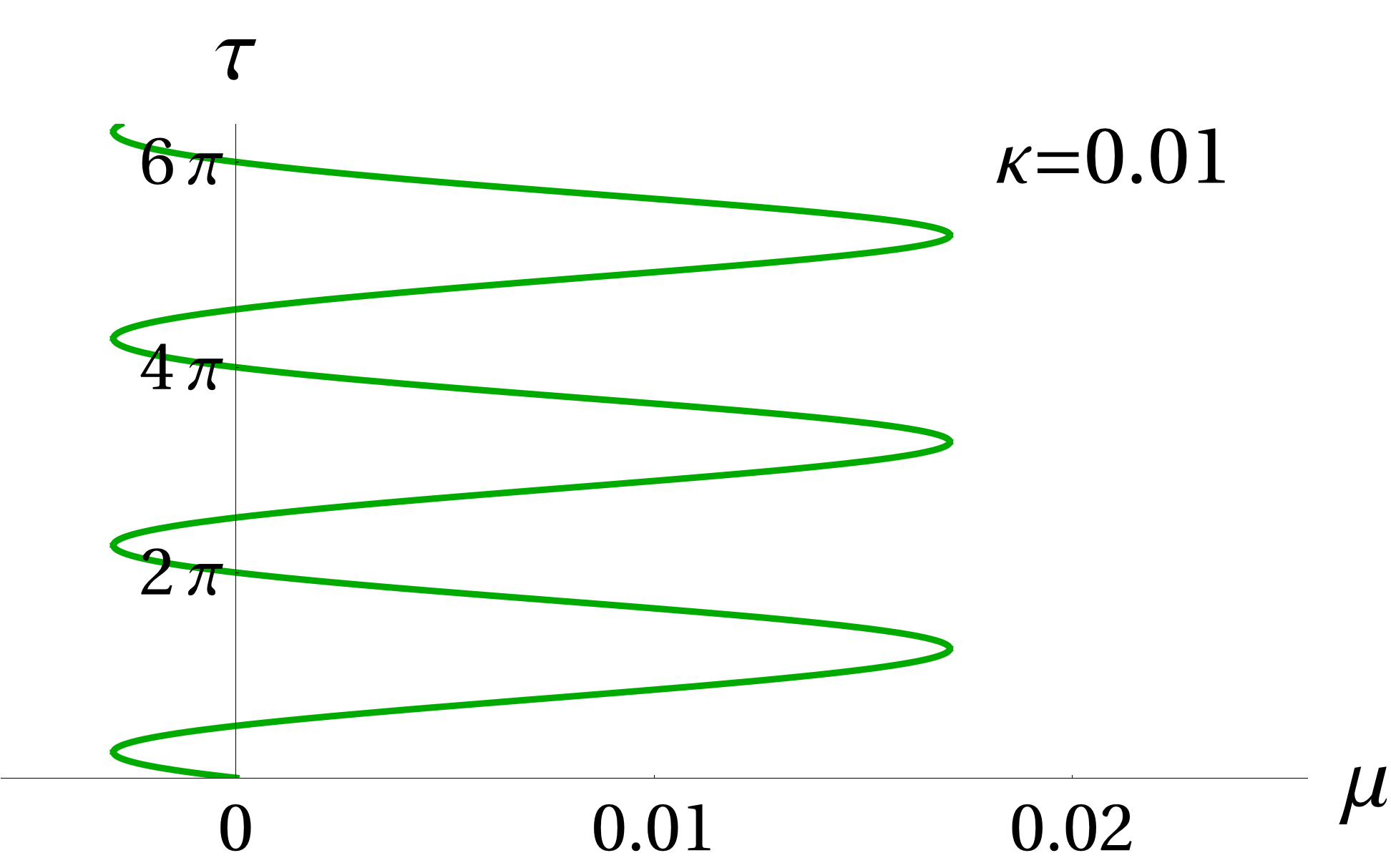}\qquad
\includegraphics[width=0.3\textwidth]{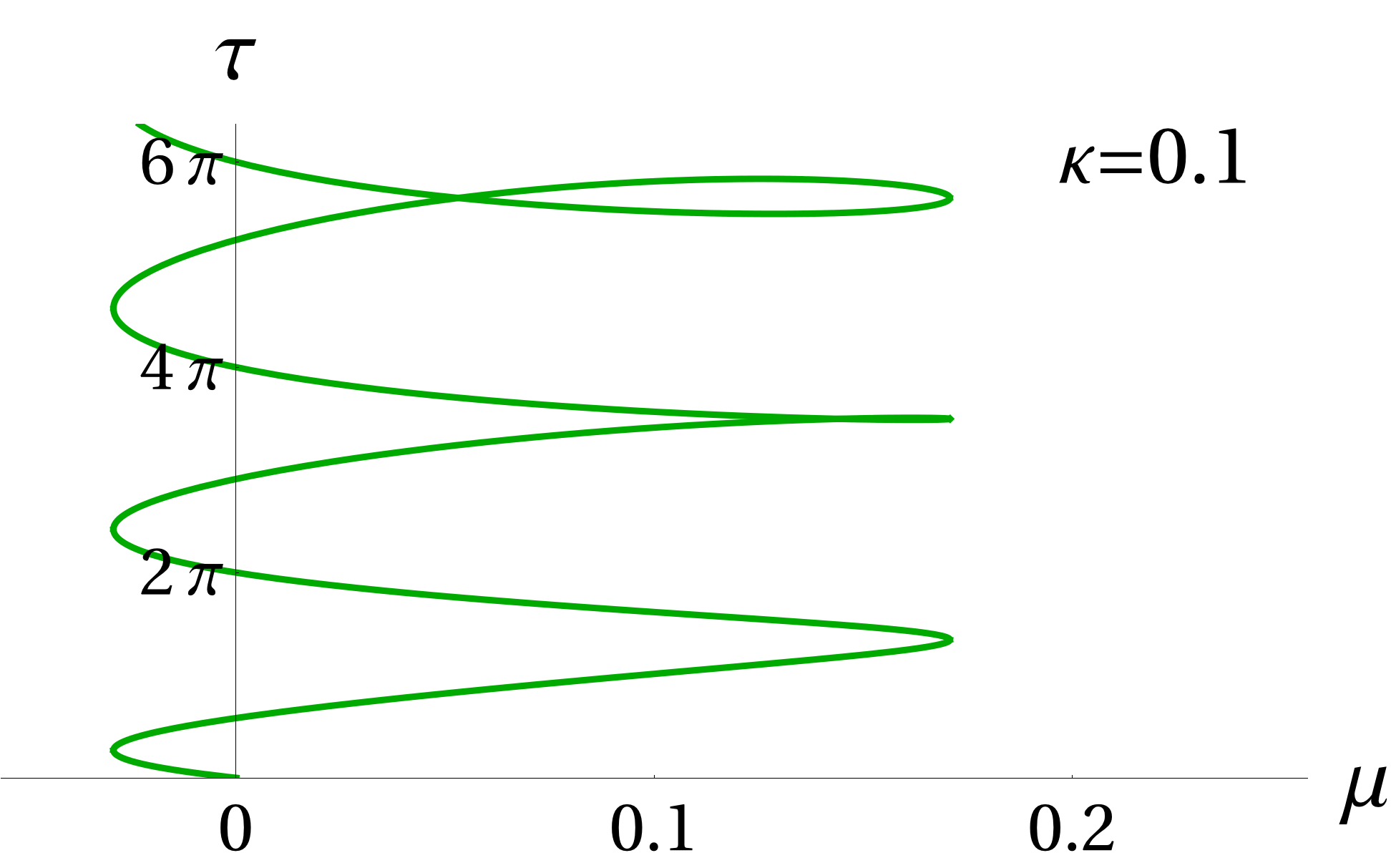}
\caption{Curves of possible Andronov-Hopf points for indicated values of $\kappa$. } \label{hopfjust}
\end{center}
\end{figure}
We consider an approximation of periodic solutions of order $2q,$ with $q>1,$ and use the steps in Table \ref{mfiter} to find a $q$-th  order bifurcation equation in local coordinates.

The eigenvectors associated with $\hat{\lambda}$ are $v=1/\sqrt{2}(1,-i)^T$ and $w= 1/\sqrt{2}(1,i)^T.$

In step 1, we define $a_1= v \theta.$ For $k=1,$ because of the form of the function $g,$ $c_j=0$ for $j\neq1$ and $c_1=-2 (1+i \gamma) \theta^3 v$, so the equations 
considered in steps 1.1 and 1.2 are
\begin{eqnarray}
L_j a_j = 0, \qquad \text{if} \quad j= 0,2, \\
(I-Q)L_1 v^{\bot} + (I-Q) G(i\omega, \mu,\tau) (-2 (1+i \gamma)) \theta^3 v = 0.
\end{eqnarray}
respectively.
From step 1.1 we obtain $a_{02}=a_{22}=0,$ and considering the step 1.2, it results $a_{1,3}=0.$ As we can observe, coefficient $a_1$ does not have $\theta^3$ term, so it has not corrections. Continuing the calculations in iterative form up to order $q,$ since $c_j=0$ for $j>1,$ all remain coefficients vanish. Therefore, we have $a_1= v \theta$ and $a_j=0,$ for $0\le j\le 2q,$ $ j\neq 1.$

Following step 2 and replacing the vector $a=(0,\ldots,0,\bar{v}\theta,0,v \theta,0,\ldots,0)$ in equation \eqref{eqcoor} results
\begin{equation}
(\hat{\lambda}(i\omega,\mu,\tau)+1)\theta + w^T G(i\omega, \mu,\tau)(-2 (1+i \gamma)) \theta^3 v = 0.
\end{equation}
Taking the coefficients of $\theta$ we observe that $\xi_k=0$ for $k=2,\ldots,q.$ Then, the $q$-th order bifurcation equation in frequency domain in local coordinates is given by
\begin{equation}\label{dfbifjust}
\hat{\lambda}(i \omega ,\mu,\tau)+1-\frac{2(1+i\gamma)}{i\omega-\mu+(1-e^{-i\omega \tau})\kappa \cos \beta}\theta^{2} =0.
\end{equation}

Simplifying the above equation and taking real and imaginary part we obtain the system
\begin{equation}\label{dfbifecjust}
\begin{array}{rcl}
-\mu +\kappa \cos \beta-\kappa \cos(\beta-\omega \tau) -2 \theta^2 & = & 0,\\
\kappa \sin \beta - \kappa \sin (\beta -\omega \tau) - 1+\omega -2\gamma \theta^2 & = & 0.
\end{array} 
\end{equation}
For fixed values of parameters $\beta$ and $\kappa,$ solutions  $(\theta,\omega,\mu,\tau)$ of the above system in a neighborhood of the critical values, are in one to one correspondence with periodic solutions of the form: $y(t) = \sqrt{2} \theta \Big( \begin{array}{c} \cos{t\omega}\\ \sin{t\omega} \end{array} \Big).$

Is important to point that, in this special example, it was not necessary to consider approximation of the coefficients as expansions in $\theta.$ Then, there is only one bifurcation equation expressed in coordinates and it is given by \eqref{dfbifjust}.
Furthermore, system \eqref{dfbifecjust} describes the dynamics of periodic solutions of any amplitude.

\subsubsection{Generalized Andronov-Hopf bifurcation}

We begin analyzing the existence of generalized Andronov-Hopf bifurcations in system. Thus, we consider the normal form \eqref{fnhopfgen} and the universal unfolding of case $q=3$ described in the above section \eqref{uuq3}. Let $\omega=\omega_0$ and $\mu=\mu_0$ values in which are verified the necessary conditions. For small values of $\theta,$ the system \eqref{dfbifecjust} allows us to calculate $\mu$ and $\omega$ as functions of $\theta$ up to any desired order. As an example, up to order $6$ we obtain
\begin{equation}\label{dfbifthjust}
\mu  =  \mu_0+\mu_1\theta^2+ \mu_2 \theta^4+\mu_3 \theta^6,
\end{equation}
with
\begin{eqnarray*}
 \mu_1 = -\left(2+\frac{2 \kappa \gamma \tau \sin(\beta-\tau\omega_0)}{1+\kappa \tau \cos(\beta-\tau\omega_0)}\right), \quad \mu_2=\frac{2 \kappa \gamma^2 \tau^2(\kappa \tau+ \cos(\beta-\tau\omega_0))}{(1+\kappa \tau \cos(\beta-\tau\omega_0))^3}, \\
 \mu_3=-\frac{4}{3}\frac{\kappa \gamma^3\tau^3 (-1+3\kappa^2 \tau^2+2\kappa \tau \cos(\beta-\tau\omega_0))\sin(\beta-\tau\omega_0)}{(1+\kappa \tau \cos(\beta-\tau\omega_0))^5},
\end{eqnarray*}
and the approximation of the frequency
\begin{eqnarray}
\omega = \omega_0 + \omega_1 \theta^2 +\omega_2 \theta^4+\omega_3 \theta^6,
\end{eqnarray}
with
\begin{eqnarray*}
 \omega_1=\frac{2\gamma}{1+\kappa \tau \cos(\beta-\tau\omega_0)},\quad \omega_2 = - \frac{2 \kappa \gamma^2 \tau^2 \sin(\beta-\tau\omega_0)}{(1+\kappa \tau \cos(\beta-\tau\omega_0))^3}\\
 \omega_3=-\frac{4}{3}\frac{\kappa \gamma^3\tau^3 (-\cos(\beta-\tau\omega_0)+\kappa \tau(-2+ \cos2(\beta-\tau\omega_0)))}{(1+\kappa \tau \cos(\beta-\tau\omega_0))^5}.
\end{eqnarray*}

Using singularity theory developed in the above section, analysis of equation \eqref{dfbifthjust} allows us to study and determine different regions in parameter space in which the system has degenerated bifurcations associated with multiplicity of periodic solutions.

Fixed $\beta=\pi/4$ and $\gamma=-10,$ we obtain the bifurcation diagram in $\kappa$-$\tau$ space showed in Figure \ref{fig_dfjust} left, where we plot transition varieties defined in \eqref{vtq3}.
At point  $(\kappa_0,\tau_0)=(-0.0475468061, 2.0927529542),$
the three curves intersect and we have the normal form
\begin{equation}
\mu_3\theta^6- (\mu-\mu_0)=0,
\end{equation}
being $\mu_3<0.$ 
In each region between this curves we find different scenarios which correspond to the universal unfolding of the above normal form. We show in Figure \ref{fig_dfjust} right, examples of each persistent bifurcation diagrams. Stability of cycles is not indicated in the diagrams, however this stability can be found considering that the equilibrium is asymptotically stable for values $\mu<\mu_0.$ Periodic solutions of small amplitude agree with numeric results (not shown). 

\begin{figure}
\begin{center}
	\includegraphics[width=.35\textwidth]{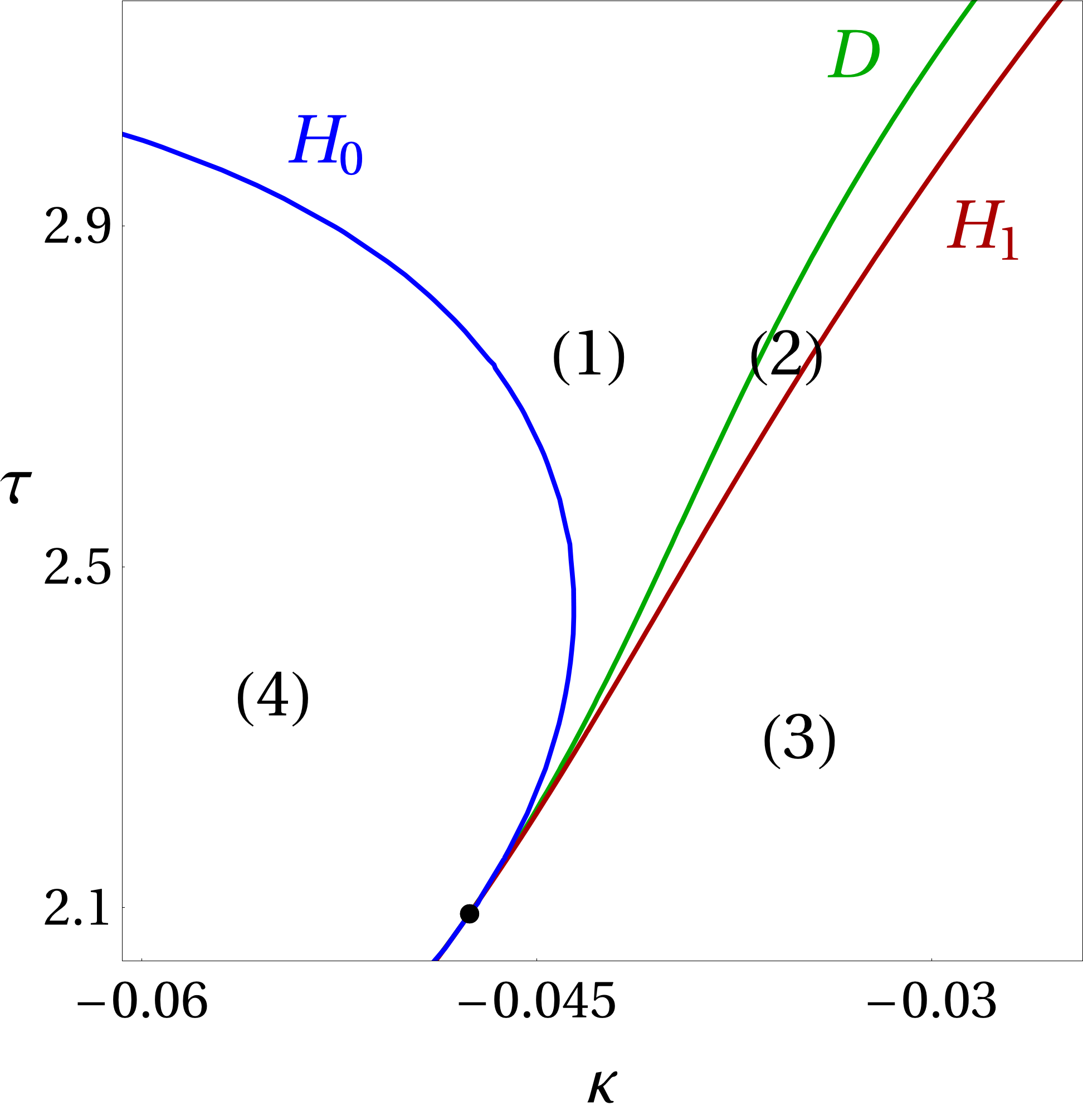}
\hspace{1cm}
	\includegraphics[width=.4\textwidth]{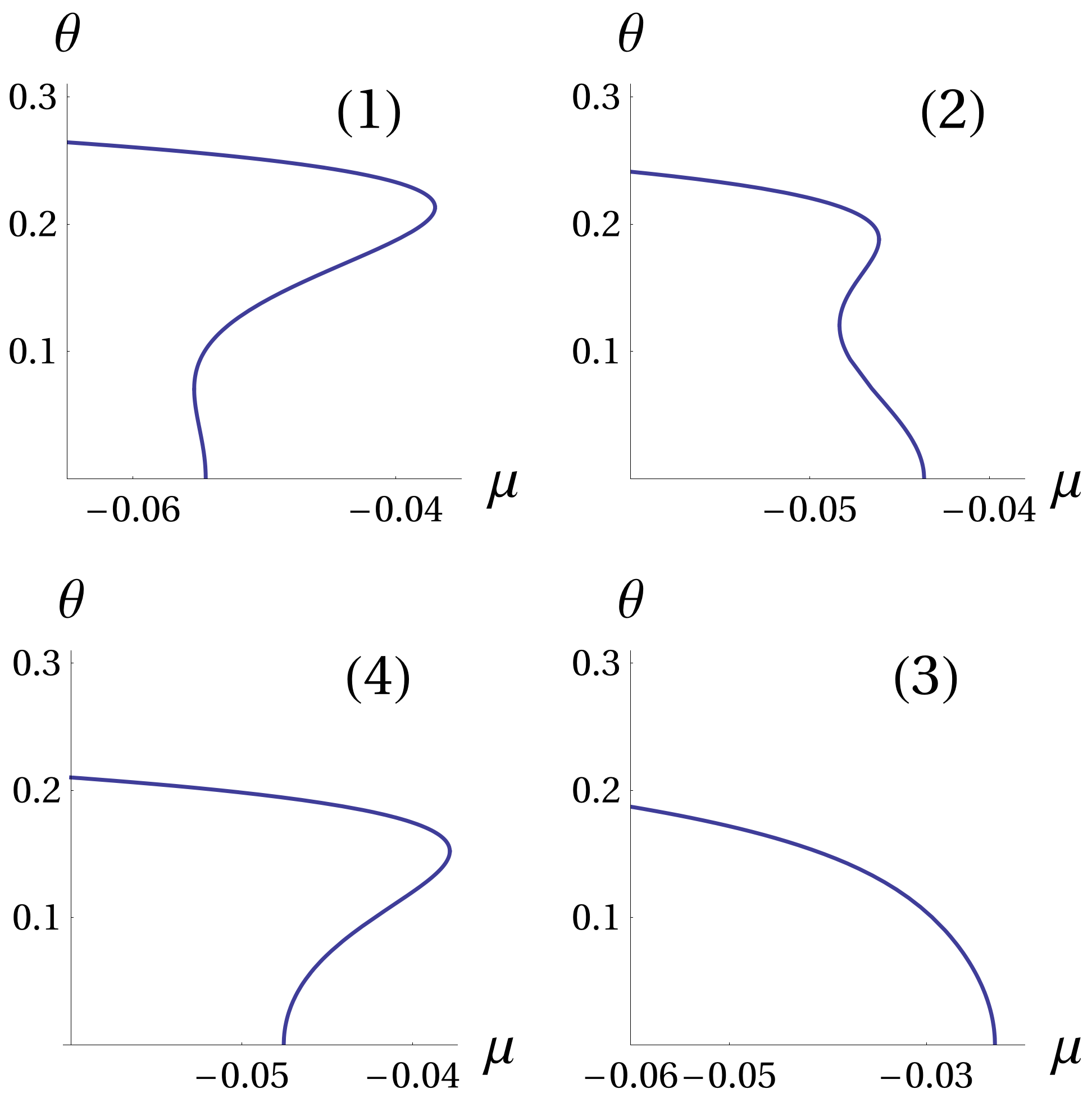}
\caption{Left: Bifurcation diagram in $(\kappa, \tau)$ space with curves defined in \eqref{vtq3}. Right: Persistent diagrams in each region indicate in figure on the left.} \label{fig_dfjust}
\end{center}
\end{figure}

\subsubsection{Cycles connecting Andronov-Hopf points}

As pointed out before, by increasing values of $\kappa$, there are more than one Andronov-Hopf bifurcation point for the same value of delay $\tau.$ This situation is better studied if we consider the alternative form to analyze the bifurcation equation presented in the above section. In particular, considering the frequency as parameter we can study connections between Andronov-Hopf bifurcation points.

We consider the bifurcation equation \eqref{dfbifecjust}, taking $z=\theta^2$ results
\begin{equation}\label{snjust}
 \begin{array}{rcl}
z & = & \displaystyle \frac{\omega-1 +\kappa \sin \beta - \kappa \sin (\beta -\omega \tau) }{2 \gamma},\\
\mu & = & \kappa \cos \beta-\kappa \cos(\beta-\omega \tau) -\frac{\omega-1 +\kappa \sin \beta - \kappa \sin (\beta -\omega \tau) }{\gamma}.
\end{array}
\end{equation}

The conditions to obtain the normal form $\epsilon z + (\mu-\mu_0)^2=0$ are
\begin{equation}
 \begin{array}{l}
 1+\kappa \tau \cos(\beta-\omega_0\tau)=0,\\
 \kappa \tau^2 \sin(\beta-\omega_0\tau)\neq0,\\
 1+\kappa \tau \cos(\beta-\omega_0\tau)+\gamma\kappa\tau\sin(\beta-\omega_0\tau)\neq0,
 \end{array}
\end{equation}
being $\omega_0$ and $\mu_0$ the critical values. To determine points that verify these conditions, as in the previous subsection, we take fixed values of $\beta$ and $\gamma,$ and consider $\kappa$ and $\tau$ as auxiliary parameters.

For fixed values $\beta=\pi/4$ and $\gamma=-10,$ we show in Figure \ref{fig_dfjust2} left, curves of points in $\kappa$-$\tau$ in which are verified the above conditions. The universal  unfolding in this case is
\begin{equation*}
 \epsilon z + (\mu-\mu_0)^2+\epsilon_0=0.
\end{equation*}
In Figures \ref{fig_dfjust2} (a) and (b), we show examples of the persistent diagrams of this type of bifurcation. We consider two fixed values of delay, $\tau=4$ and $\tau=6,$ and the indicates values of $\kappa.$ In particular, in the Figure \ref{fig_dfjust2} (a) we can see the existence of several branches of periodic solutions that connect the equilibrium, by increasing $\kappa$ this branches disappear. The equilibrium and periodic solutions are unstable in the region of space $\kappa$-$\tau$ considered, and for the range of bifurcation parameter $\mu$ observed. We only shows periodic solutions associated with the studied bifurcation, in all cases there was observed another branch of periodic solutions (not shown) that emerge from a different Andronov-Hopf bifurcation point.

\begin{figure}[ht!]
\begin{center}
\begin{minipage}{4cm}
	\includegraphics[width=1\textwidth]{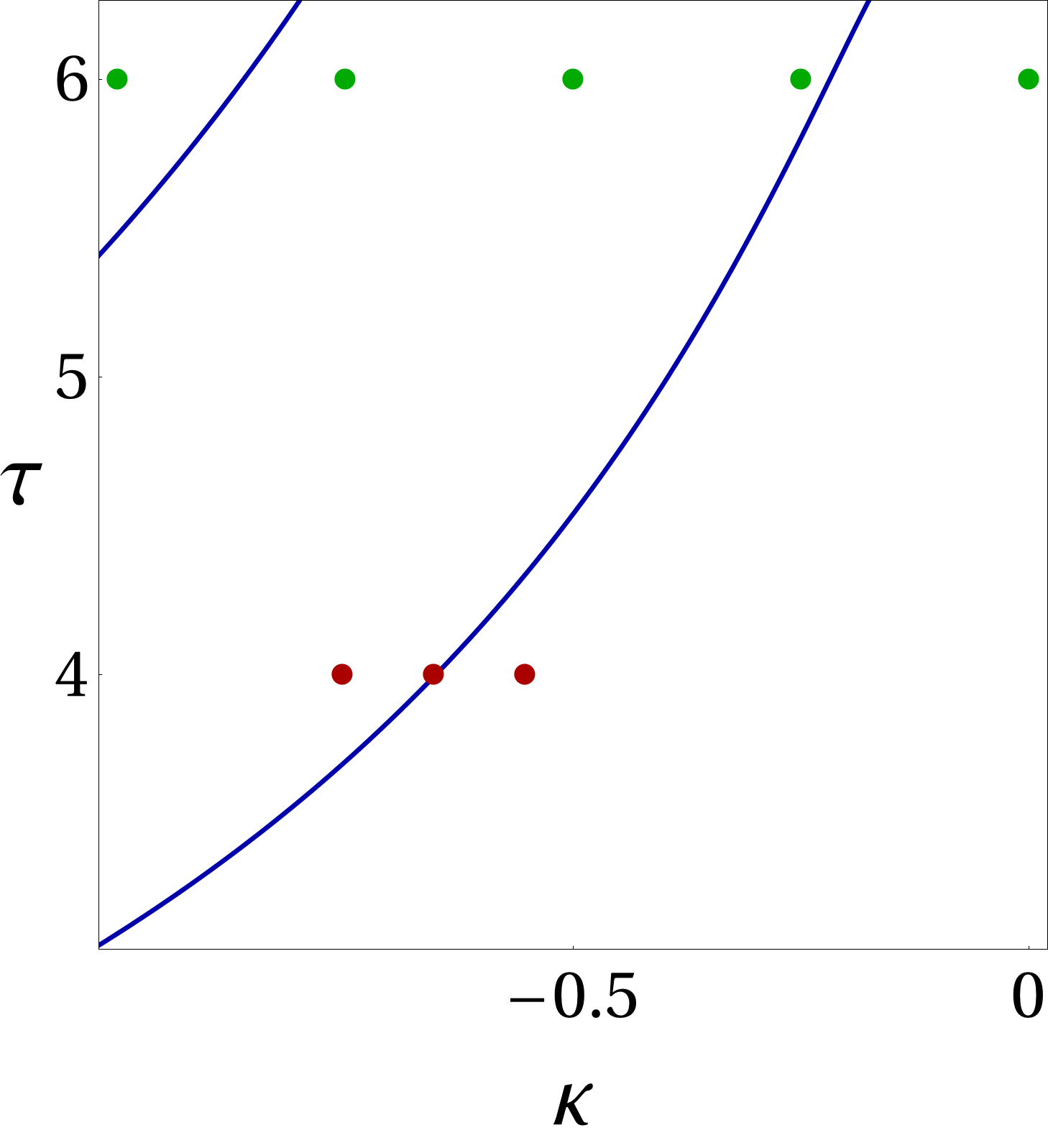}
\end{minipage}\hspace{1.5cm}
\begin{minipage}{7cm}
\centering
	\includegraphics[width=1\textwidth]{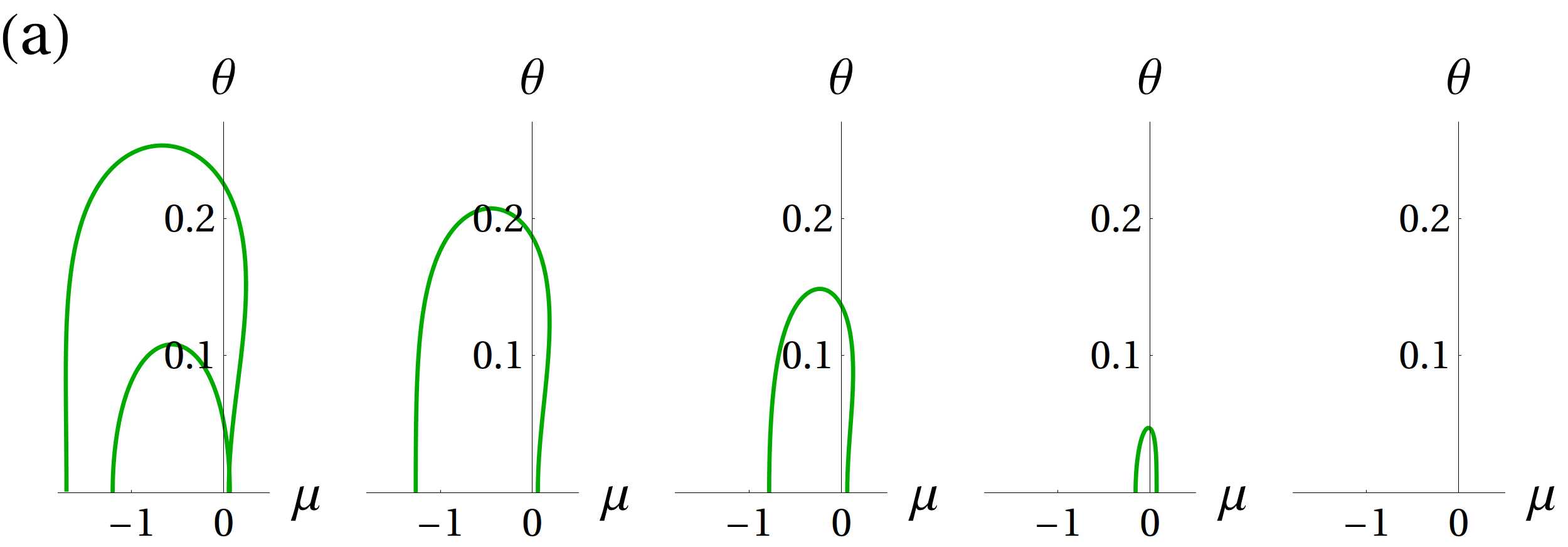}
        \includegraphics[width=1\textwidth]{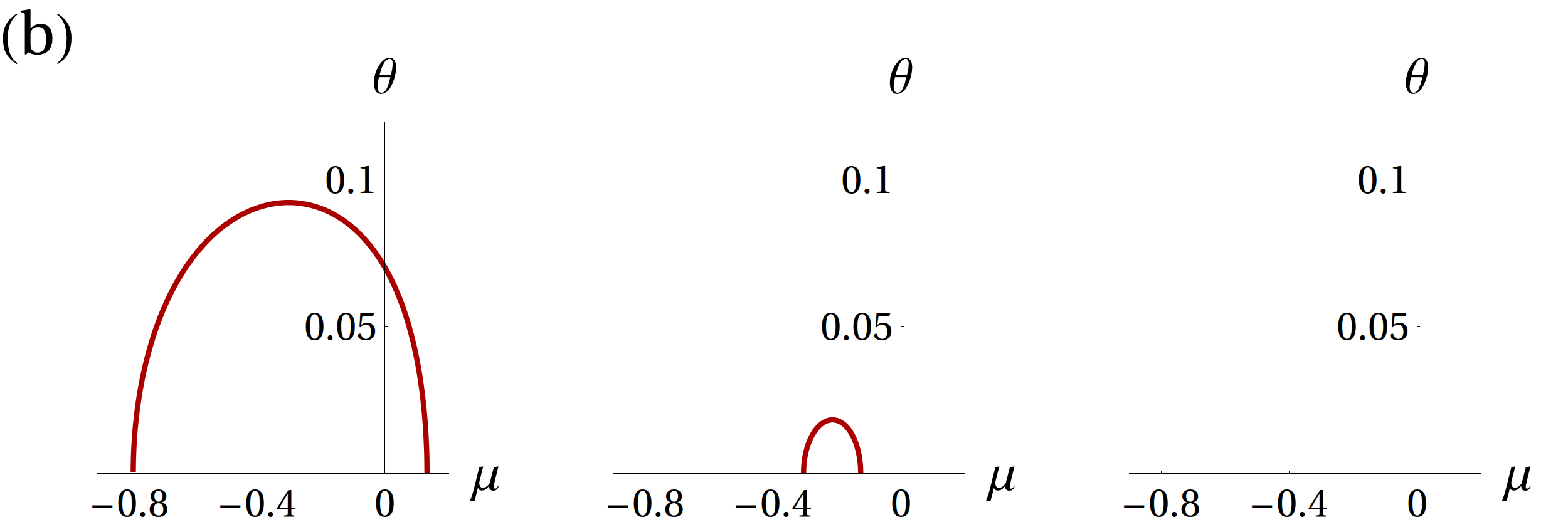}
\end{minipage}
\caption{Left: Curves of points that verify normal form with $p=2$. Right: (a) Diagrams for $\tau=6.$ From left to right we consider $\kappa= -1, -0.75, -0.5, -0.25, 0$. (b) Diagrams for $\tau=4.$ From left to right we consider $\kappa=-.753, -0.653, -0.553.$   } \label{fig_dfjust2}
\end{center}
\end{figure}

Now we want to find values of the parameters that verifies the conditions to obtain the normal form $\epsilon z + (\mu-\mu_0)^3=0.$ We observed that values that verify $\frac{dz}{d\omega}(\omega_0,\rho)=0$ and $\frac{d^2z}{d\omega^2}(\omega_0,\rho)=0,$ also satisfy $\frac{d\mu}{d\omega}(\omega_0,\rho)=0.$ Thus, the bifurcation at these values has a bigger codimension and the existence of periodic solutions associated with it can not be described performing the study developed in the above section.

\subsection{Bautin bifurcation in a first order delay differential equation}

We consider the scalar delay differential equation
\begin{equation}\label{leuk}
x'(t)= -\delta \, x(t) -\beta \left(\frac{x(t)}{1+(x(t))^n}-\frac{k x(t-\tau)}{1+(x(t-\tau))^n}\right),
\end{equation}
where $\delta>0$ is the bifurcation parameter, the auxiliary parameters are $\beta>0,$ $n>0$ and $k>0,$ and $\tau$ is a positive constant delay. This delay differential equation comes from a model of periodic chronic myelogenous leukemia \cite{mackey04}. The system \eqref{leuk} was study using normal forms in \cite{anca13}.

For all parameter values, $x=0$ is an equilibrium of \eqref{leuk}. There is another equilibrium given by
\begin{equation}
\hat{x}=\left(\frac{\beta}{\delta}(k-1)-1\right)^{\frac{1}{n}}.
\end{equation}
A full study of the equilibrium stability was presented in \cite{anca13}. There, it was proved that the equilibrium $\hat{x}$ is well defined and takes significant values for the model only if $1<k<2.$
In the next, we study periodic solutions associated to $\hat{x}.$

Consider the minimal realization $(A_0,A_1,B,C)=(-(\delta+1),0,1,1),$ and the non-linear function
\begin{equation}
g(x(t),x(t-\tau))= - x(t)+ \beta \left(\frac{x(t)}{1+(x(t))^n}-\frac{k x(t-\tau)}{1+(x(t-\tau))^n}\right).
\end{equation}
The linear transfer function results
\begin{equation}
GJ(s,\delta,\tau)= -\left(1 + \frac{\delta\left((n-1)(k-1)\beta-n\delta\right)}{\beta (k-1)^2} (1-k e^{-s\tau})\right)\frac{1}{s+\delta+1}.
\end{equation}
There is only one characteristic function $\hat{\lambda}(s,\delta,\tau)=GJ(s,\delta,\tau).$%

If we define $\beta_1=\delta((n-1)(k-1)\beta-n\delta)/(\beta (k-1)^2),$ the condition in Lemma \ref{lemma1} (i.e., $\hat{\lambda}(i\omega,\delta,\tau)=-1)$, leads to equations
\begin{equation}
\begin{array}{ccl}
\delta & = & \displaystyle \beta_1 \left(1-k \cos \omega \tau\right),\\
\omega & = & \displaystyle  \beta_1 k \sin \omega \tau .\\
\end{array}
\end{equation}
From the above system it results 
\begin{equation}\label{tau}
\omega =\sqrt{(\beta_1 k)^2-(\delta-\beta_1)^2},\quad\text{and}\quad\tau=\frac{1}{\omega}\arccos \left(\frac{\beta_1-\delta}{\beta_1 k}\right).
\end{equation}
For fixed values of the parameters $\beta,$ $k$ and $n,$ the above equation defines a curve of possible Andronov-Hopf points in the $\delta$-$\tau$ space.

We apply the algorithmic process described in Table \ref{mfiter} to obtain the bifurcation equation and to study the existence of periodic solutions. On the one hand, since $\dim(GJ)=1,$ the eigenvectors are $v=w=1,$ which simplify the calculations. But on the other hand, the non-linearities of the function $g$ brings us to long and complicated calculations, even for small order $q.$

Since we are interested in Bautin bifurcations we consider the algorithmic process of order $q=2.$ Unlike the previous example we will not discuss the details of the calculation and we compare the calculations with numeric results. From now one we consider fixed $n=2,$ $\beta=2.5,$ and several fixed values of $k\in (1,2).$

\subsubsection{Bautin bifurcation}
For fixed values of $\tau,$ let $\omega_0$ and $\delta_0$ be critical values such as the hypotheses of Theorem \ref{teoppal} are verified. If the condition \linebreak $\left.\left(\real{\lambda_{\delta}} \imag{\lambda_{\omega}}-\imag{\lambda_{\delta}} \real{\lambda_{\omega}}\right)\right|_{(\omega_0,\delta_0,\tau)}\neq0$ is verified at the considered critical values, then we can calculate the expression of order $q=2$ of $\delta$ as function of $\theta$ and the auxiliary parameters. Using this expression we study the dynamic of the small amplitude periodic solutions of the system \eqref{leuk}.

First we consider $k=1.5.$ For fixed values  $\delta_0=0.1100351576$ and $\tau_0= 4.9740704569,$ and the critic frequency $\omega_0=0.2624792103,$ we have the normal form
\begin{equation}\label{bautinnf}
\delta_2\theta^4-(\delta - \delta_0)=0.
\end{equation}
being $\delta_2=0.0019537383.$ Since the coefficient $\delta_2$ is positive it follows that in a neighborhood of the critical values (considering the universal unfolding \eqref{bautinuu}), the systems has  a branch of stable limit cycles if $\delta_1>0$ and it presents a fold bifurcation of cycles if $\delta_1<0$. From the calculated expression for $\delta_1$ we obtain $\delta_1>0$ ($\delta_1<0$) if $\tau<\tau_0$ ($\tau>\tau_0$).

In Figure \ref{leukk1p5} left, we show the curve of Andronov-Hopf points in the $\delta$-$\tau$ space, the equilibrium is stable in the shaded region below that curve. The black dot represents the Bautin bifurcation point $(\delta_0,\tau_0).$ In Figure \ref{leukk1p5} center and right, we plot persistent bifurcation diagrams. We compare the amplitudes of approximated periodic solutions obtained with algorithmic process of order $2$ (solid line) and the numerical calculations obtained with DDE-BIFTOOL \cite{ddebiftool, ddebmanual} (dotted line). We can observe the good agreement of both results.
\begin{figure}
\begin{center}
 \includegraphics[width=.3\textwidth]{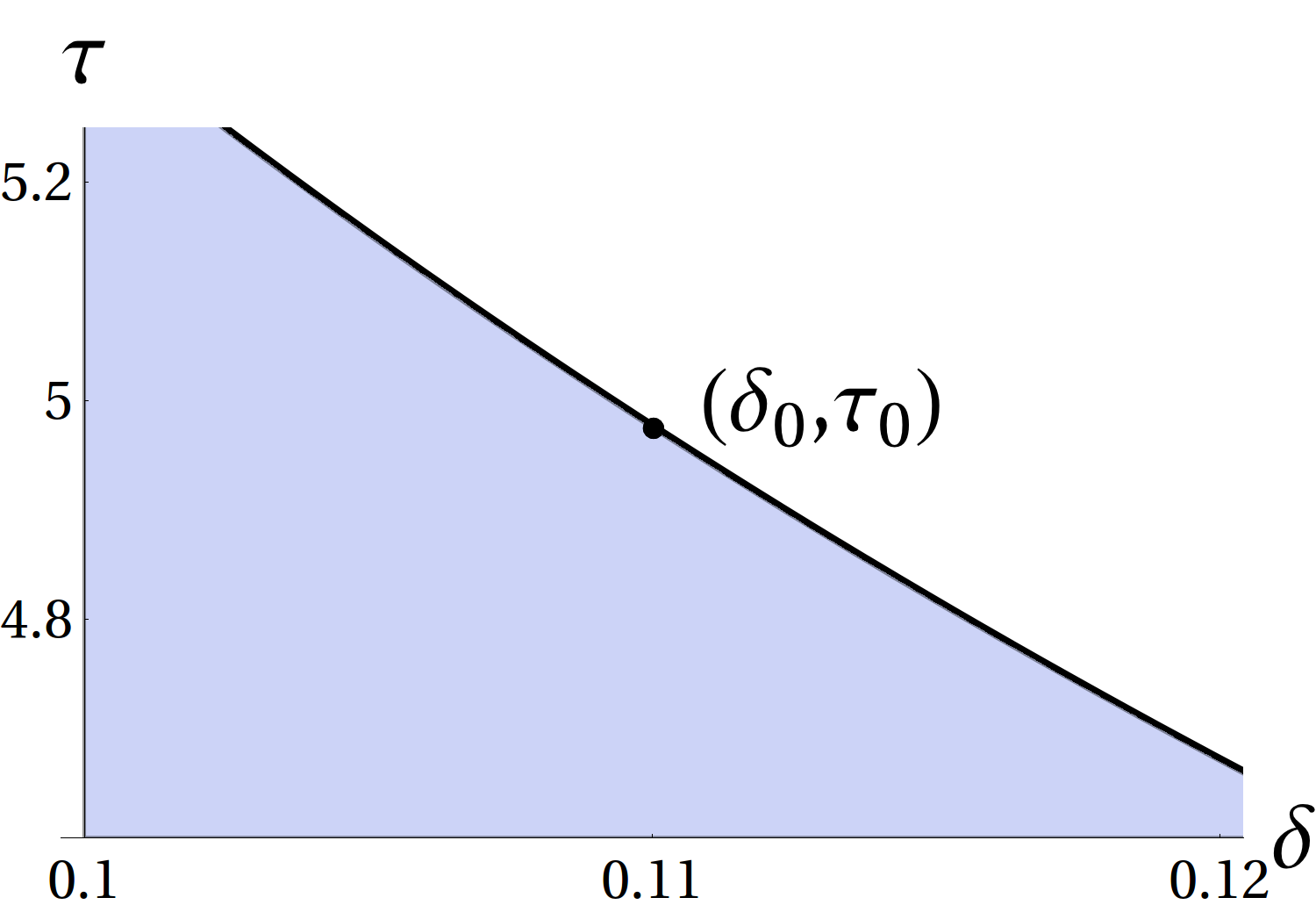}
\includegraphics[width=.3\textwidth]{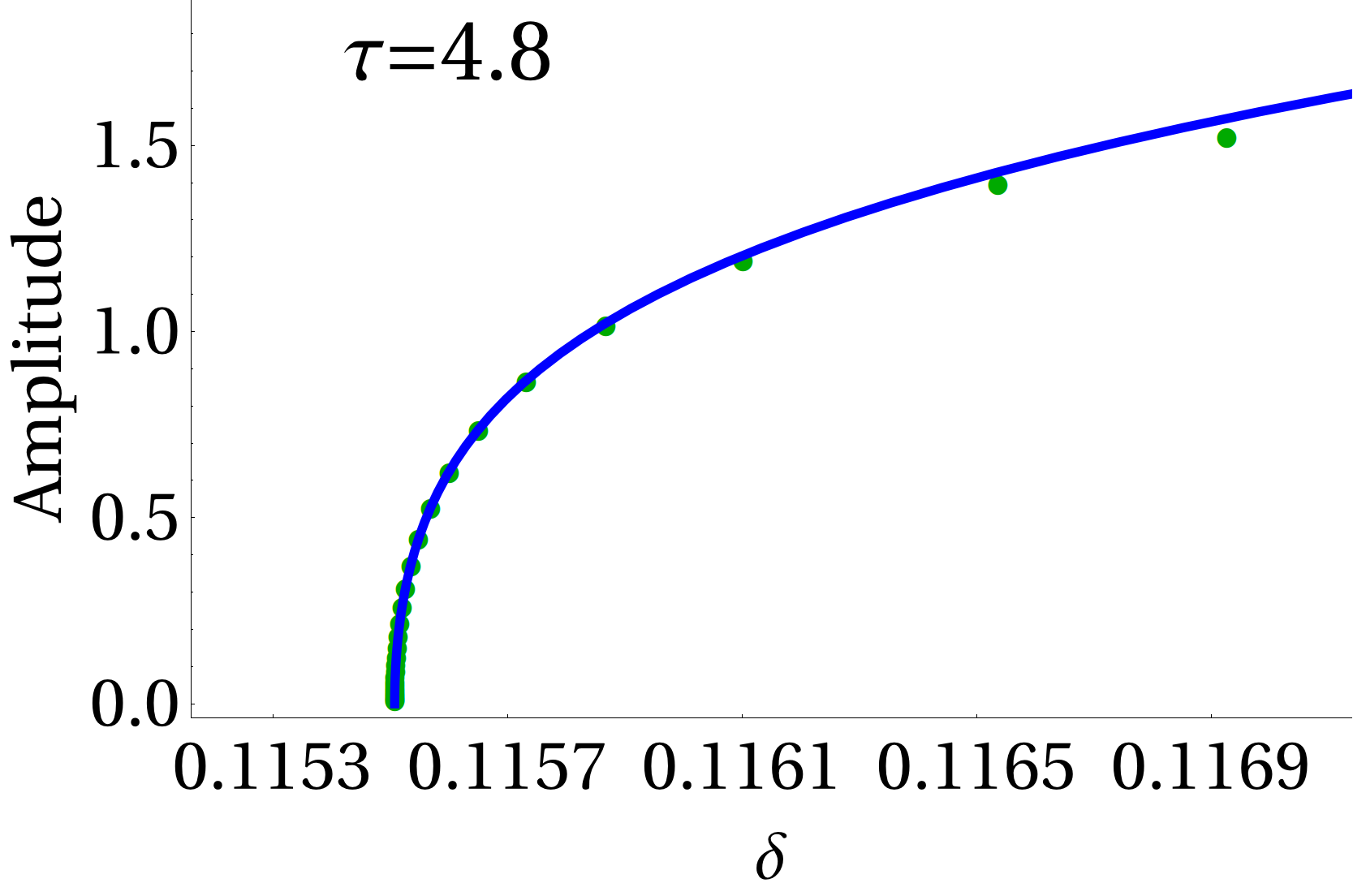}
\includegraphics[width=.3\textwidth]{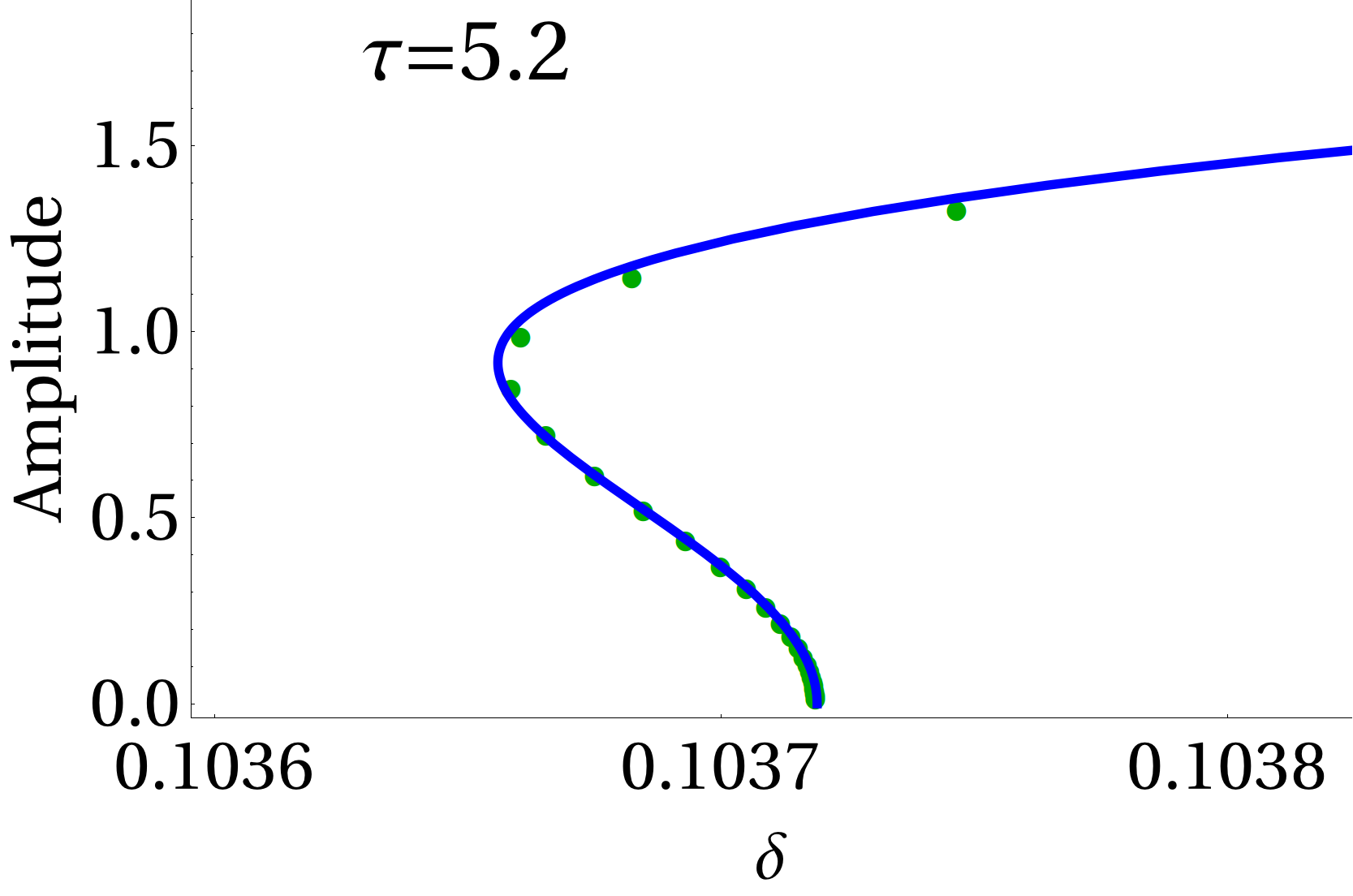}
\caption{Left: Adronov-Hopf point curve. Black dot: Bautin bifurcation point. Center and Right: Amplitudes of periodic solutions using frequency domain (solid line) and DDE-BIFTOOL (dotted line) taking values of $\tau$ indicated in figures.} \label{leukk1p5}
\end{center}
\end{figure}

As another example we take $k=1.01.$ For fixed $\delta_0=0.0023073665$ and $\tau_0=5.301432998$ and the critic frequency $\omega_0=0.0396791,$ we obtain the normal form \eqref{bautinnf} with $\delta_2=0.0000417833.$ Since $\delta_2>0$ the dynamic observed in the universal unfolding  is similar to the previous case (where $k=1.5$). But, as we show in Figure \ref{leukk1p01}, the amplitude of periodic solutions obtained with the algorithmic process (solid line) agree with the amplitude of numeric solutions calculated with DDE-BIFTOOL (dashed line) when the amplitude is small. As we can observe, the fast increment of the amplitude of the numeric solutions indicates the existence of a global bifurcation that brings to a limit cycle of big amplitude (canard explosion). To illustrate this situation we show in Figure \ref{profk1p01} the profiles of the three cycles (with normalized period), obtained numerically for fixed $\tau=5.35$ and $\delta=0.002280345$.

\begin{figure}
\begin{center}
\includegraphics[width=.3\textwidth]{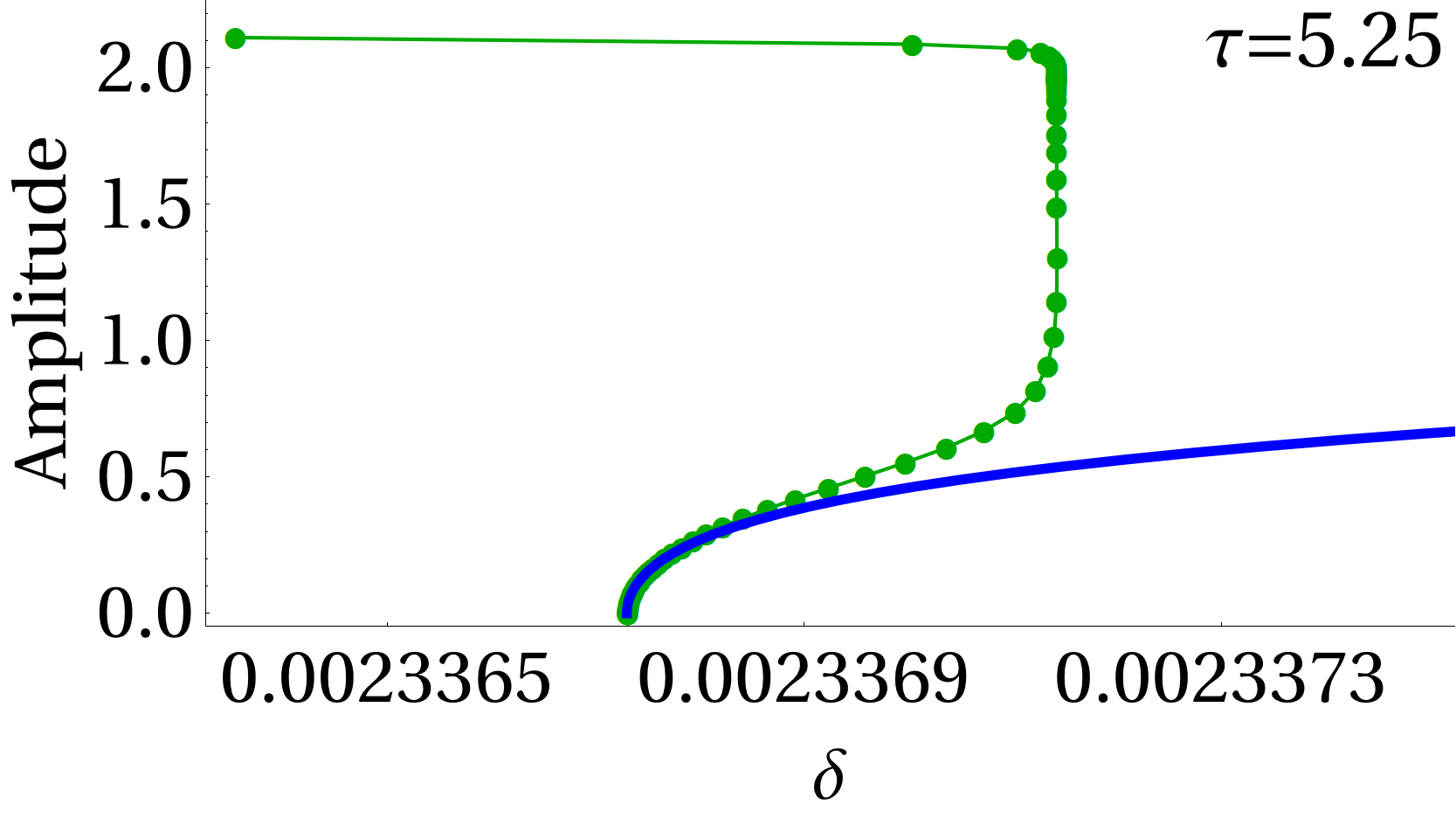}  \hspace{.2cm}
\includegraphics[width=.3\textwidth]{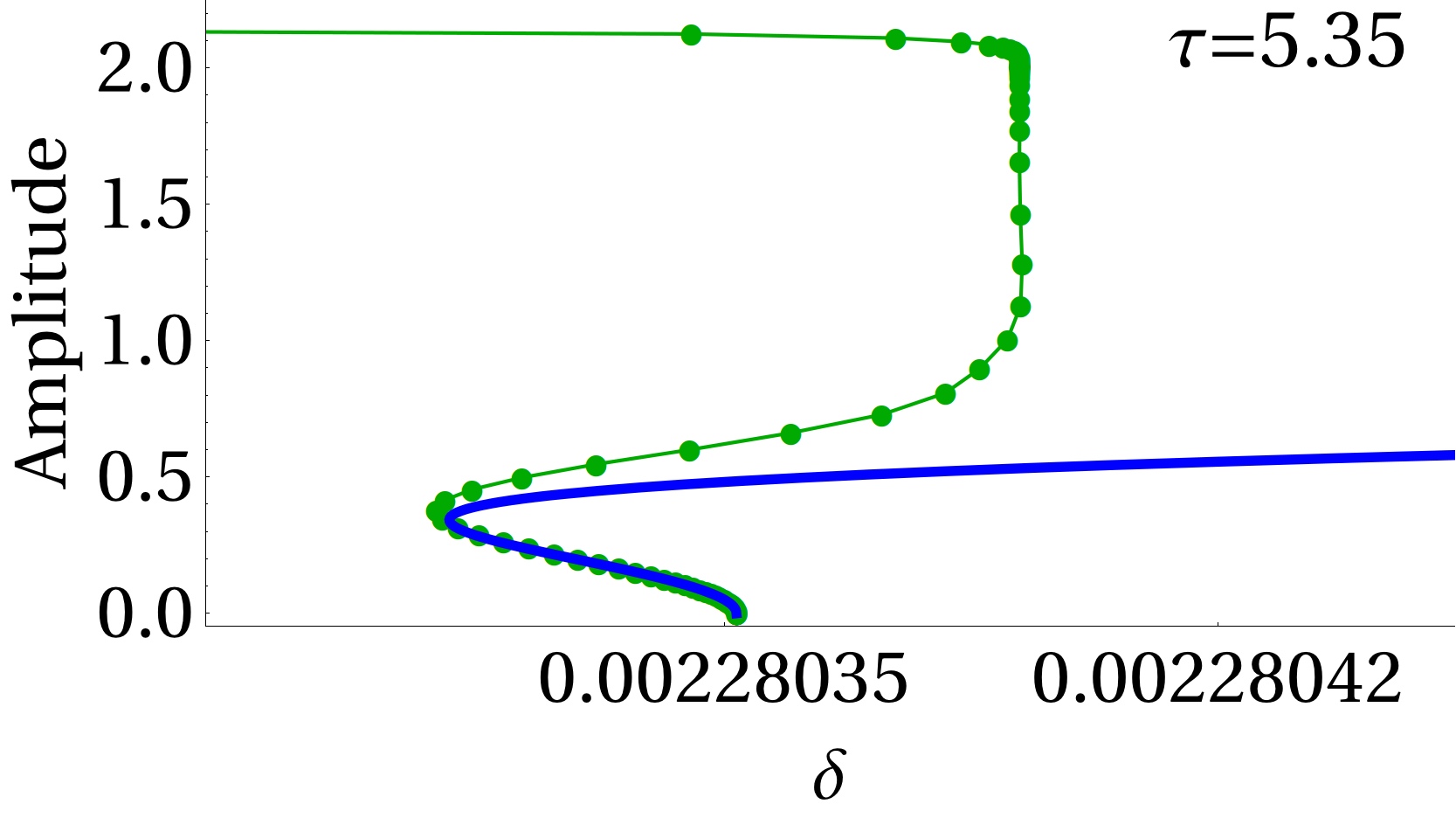}
\caption{Amplitudes of the periodic solutions with frequency domain (solid line) and DDE-BIFTOOL (dotted line) taking the values of $\tau$ indicated in the figures.} \label{leukk1p01}
\end{center}
\end{figure}
\begin{figure}
\begin{center}
\includegraphics[width=.3\textwidth]{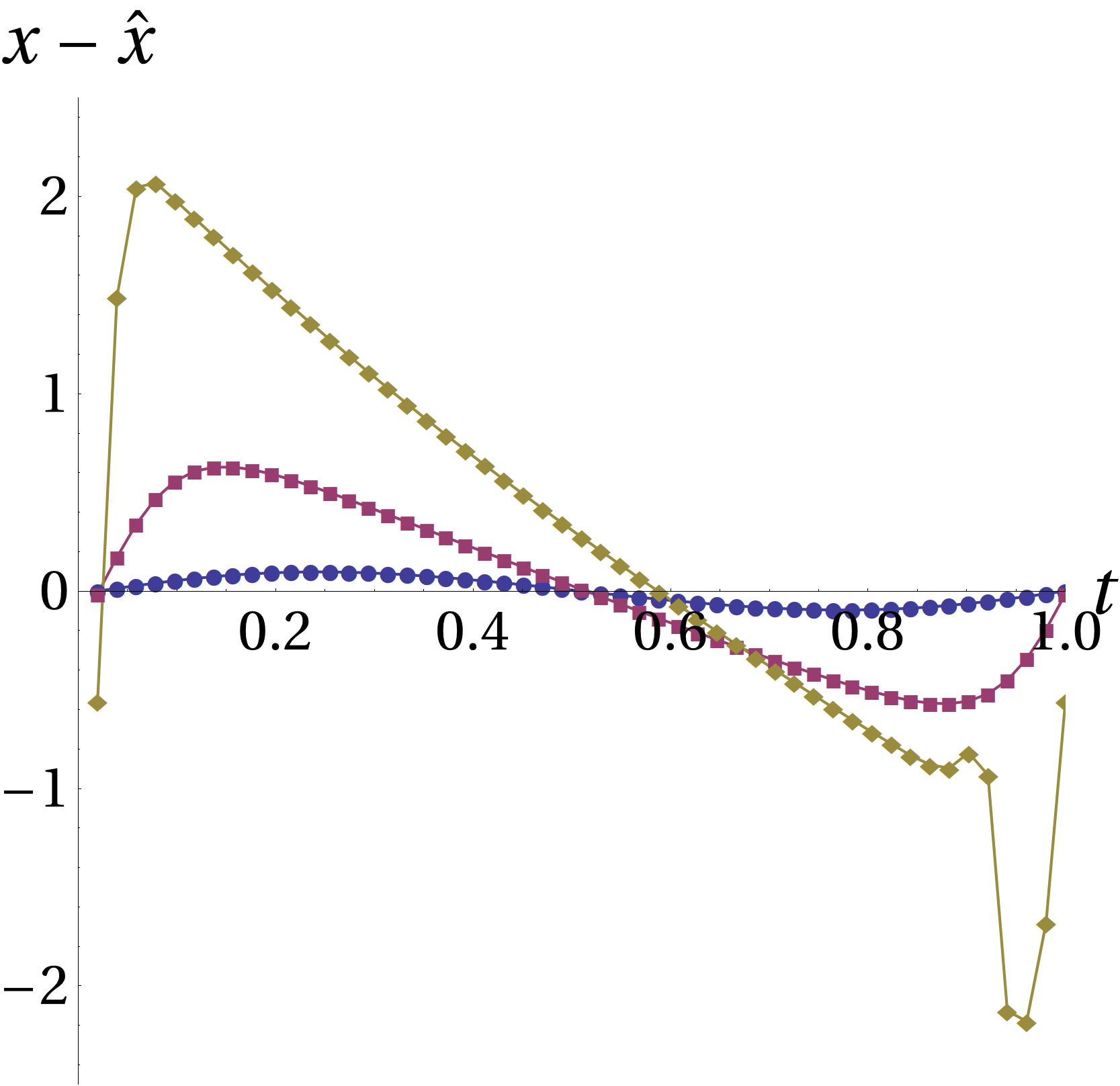}
\caption{Profiles of the numeric solutions for fixed $\delta = 0.002280345$ in Figure \ref{leukk1p01} right.} \label{profk1p01}
\end{center}
\end{figure}

\section{Conclusions}

In this work we present an approach based on frequency domain methods, to study local periodic solutions in delay differential equations which improves the existent results. In the main result of this paper we obtain a bifurcation equation for local oscillations and an expression of periodic solutions up to any fixed order. To calculate these expressions we propose an algorithmic process obtained from the proof of the main theorem. 

The cualitative behaviour of the periodic solutions is studied analyzing the bifurcation equation with singularity theory. We obtain conditions to determine all bifurcations for periodic orbits with codimension less than or equal to two. 

We show the potentiality of the proposed algorithmic approach with two examples. The first one is a time-delayed feedback system which is well-know in  applications of the control theory. This example has a rich dynamic and allows us to show how to determine different scenarios in which multiplicity of periodic solutions is observed. The second example is a first order delay differential equation where we determine a Bautin bifurcation. We compare the numeric calculations with the analytical approximations. We observe that the latter determines with great precision the smallest limit cycles. Besides, the numerics results indicate that this system has a canard-like explosion for some values of the delay.

\subsection*{Acknowledgments}

The work is supported by the Universidad Nacional del Sur (Grant no. PGI 24/L096).


\end{document}